\documentclass[12pt]{iopart}

\expandafter\let\csname equation*\endcsname\relax
\expandafter\let\csname endequation*\endcsname\relax
\usepackage[T1]{fontenc}
\usepackage{amsmath,amssymb,amsthm}
\usepackage{graphicx}
\usepackage{bm}
\usepackage{booktabs}
\usepackage{array}
\usepackage{pgfplots}
\usepackage{tikz}
\usepackage{placeins}
\usetikzlibrary{arrows.meta,positioning,calc}
\pgfplotsset{compat=1.17}
\usepackage{iopams}
\usepackage{hyperref}
\usepackage{url}

\hypersetup{
  hidelinks,
  pdftitle={Uniqueness and reconstruction of reciprocal constitutive generators from equilibrium response maps},
  pdfauthor={Jonathan Washburn and Megan Simons},
  pdfsubject={Inverse problems; structural identifiability; reconstruction; equilibrium response maps},
  pdfkeywords={inverse problem, structural identifiability, reconstruction, equilibrium response map, inverse optimization, Schroeder equation, stability, reciprocal generator}
}

\theoremstyle{plain}
\newtheorem{theorem}{Theorem}
\newtheorem{corollary}[theorem]{Corollary}

\newtheorem{proposition}{Proposition}
\theoremstyle{definition}
\newtheorem{definition}{Definition}

\begin{document}

\title[Reciprocal generator reconstruction]{Uniqueness and reconstruction of reciprocal constitutive generators from equilibrium response maps}
\author{Jonathan Washburn$^1$\footnote{ORCID: 0009-0001-8868-7497} and Megan Simons$^1$\footnote{Corresponding author. ORCID: 0000-0001-9457-7019}}
\address{$^1$ Recognition Science, Recognition Physics Institute, 1401 Lavaca St \#41100, Austin, TX 78701, USA}
\ead{msimons@recognitionphysics.org}
\date{23 October 2026}
\begin{abstract}
We recover a shared constitutive generator from equilibrium locations on
weighted graphs. For an even $C^2$ log-coordinate
potential with positive curvature, the complete two-site response map
determines the generator up to positive scale. A sweep over all positive
boundary values at one fixed, possibly unknown, unequal weight ratio gives
global identification through a Schr\"oder--Koenigs iteration. We extend
this result to a three-vertex chain with two observed free vertices under a
sweep condition. For variable-ratio control, we prove finite-target
recovery, local Lipschitz stability, and noisy-bisection bounds. Fixed-ratio
stability and nonparametric estimation under uncertain controls remain open.
In the reciprocal family~$J_p$, an exact Gaussian model-selection
probability benchmarks synthetic experiments under observation noise and
errors in variables. These are identifiability results under controls; no
measured data or physical system is analyzed.
\end{abstract}


\maketitle
\section{Introduction}
\label{sec:intro-scope}

This paper asks when, and how stably, a scalar constitutive law can be
recovered from equilibrium locations alone. The setting is a finite weighted
graph whose site and edge costs share one unknown even generator~$\Psi$ of
log-ratios. The data are constrained equilibrium locations under controlled
boundary values and site-to-edge weight ratios; neither fluxes nor objective
values are observed. The answers are constructive and are summarized in
Section~\ref{sec:contributions}.

\subsection{Inverse-problem formulation}
\label{sec:formulation}
Let $\Psi:\mathbb{R}\to\mathbb{R}_{\ge0}$ be even, $C^2$, with $\Psi(0)=0$ and
$\Psi''(t)>0$ for all~$t$, and write $g:=\Psi'$. On the elementary two-site
exchange with free log-coordinate~$v$, pinned boundary~$u>0$, and weight
ratio~$r>0$, the forward response is
\begin{equation}
\label{eq:response-op}
\mathcal{R}_{\Psi}(u,r)
  \;:=\;
  \arg\min_{v\in\mathbb{R}}
  \bigl\{r\,\Psi(v)+\Psi(u-v)\bigr\}.
\end{equation}
The inverse problem is to recover~$\Psi$ (equivalently~$g$) from the map
$\mathcal{R}_{\Psi}$. An overall positive scale of~$\Psi$ leaves every constrained
minimizer unchanged and is therefore not identifiable from equilibrium locations
alone; unit log-curvature $g'(0)=1$ removes that ambiguity.

Two control designs are used. Variable-ratio designs vary both $u$ and~$r$;
each equilibrium then yields one ratio of values of~$g$, which is the basis of
finite-target recovery and of the stability estimates. Fixed-ratio designs
sweep~$u$ at one constant, possibly unknown, ratio $r\neq1$ and identify~$g$
through an iteration. Equal weights, $r=1$, give $v^\star=u/2$ for every~$\Psi$
and carry no information. On larger graphs the design also specifies which
vertices are pinned, and on the two-free-vertex chain both free locations are
observed.

\subsection{Contributions}
\label{sec:contributions}
The main results are:
\begin{enumerate}
  \item \emph{Uniqueness with an explicit inversion formula}
    (Theorem~\ref{thm:response-ident}). The complete response map gives every
    ratio $g(s)/g(w)$ through $r=g(u-v^\star)/g(v^\star)$ and hence determines
    $\Psi$ up to its unavoidable positive scale.
  \item \emph{Identification at one fixed ratio}
    (Theorem~\ref{thm:fixed-ratio}). One boundary sweep at a constant, initially
    unknown $r\neq1$ determines both $r$ and~$\Psi$ up to scale: the response
    supplies an attracting map~$T$ with
    $g\circ T=g/\max\{r,1/r\}$, and Koenigs' limit recovers~$g$.
  \item \emph{Finite-target recovery and stability}
    (Corollaries~\ref{cor:finite-query}--\ref{cor:stability},
    Proposition~\ref{prop:bisection}). We give residual and local Lipschitz
    bounds and a noisy-bisection guarantee under bounded or replicated Gaussian
    location errors; the constants depend on the unknown generator.
  \item \emph{Single-free-vertex reduction} (Corollary~\ref{cor:graph-reduction}).
    A free vertex with pinned neighbors reduces to the two-site equation with
    an effective ratio.
  \item \emph{A chain with two free vertices} (Theorem~\ref{thm:chain}). With
    both free locations observed, an explicit sweep condition and unequal leaf
    weights identify the generator without recorded boundary values or known
    weights. General multi-free-vertex graphs are not treated.
\end{enumerate}
The analytic tools are classical (strict convexity, the implicit-function and
mean-value theorems, and Schr\"oder--Koenigs linearization). The contribution is
the structured, nonparametric identification result, not a general theory of
inverse optimization or functional equations; Section~\ref{sec:related}
locates it relative to prior work.

The reciprocal cost $J(x)=\tfrac12(x+x^{-1})-1$ plays a secondary role, as an
independently motivated candidate selected by a ratio-composition
law~\cite{WashburnAczel}; none of the inverse results depends on it. Closed-form
responses for the calibrated family~$J_p$ support the synthetic experiments.

\subsection{What is proved, what is simulated, and what is not attempted}
\label{sec:levels}
We distinguish three levels of claim.
\begin{enumerate}
  \item[(a)] \emph{Structural identifiability and finite-query recovery:
    proved} (Section~\ref{sec:ident}). The variable-ratio inversion underlies
    the local error and noisy-bisection bounds. The fixed-ratio and chain
    theorems instead use an iterated-orbit limit and establish exact-data
    identification only; their noise behavior is unresolved.
  \item[(b)] \emph{Synthetic finite-data model selection: simulated}
    (Section~\ref{sec:numerics}). Within~$J_p$, a fixed $15$-condition design
    yields simulated selection probabilities, a Fisher-information bound, and
    least-squares behavior under stated noise models. These are design
    properties, not nonparametric estimation of~$g$.
  \item[(c)] \emph{Experimental reconstruction: not attempted.} No measured
    data are analyzed; Section~\ref{sec:outlook} lists what a physical
    realization would have to validate.
\end{enumerate}
Controlled equilibrium response in networked transport and biochemical
exchange, where ratios to local references are natural states and boundaries
and conductances are natural controls, motivates the formulation.

\subsection{Relation to previous work}
\label{sec:related}
Table~\ref{tab:related} summarizes the comparison developed below. For each
adjacent framework we state what is established, how the present problem
differs, and what remains open.

\begin{table}[htbp]
\centering
\caption{The present inverse problem compared with adjacent frameworks.
Entries describe typical settings, not every result in each literature.}
\label{tab:related}
\footnotesize
\begin{tabular}{@{}>{\raggedright\arraybackslash}p{0.19\linewidth}
                >{\raggedright\arraybackslash}p{0.22\linewidth}
                >{\raggedright\arraybackslash}p{0.23\linewidth}
                >{\raggedright\arraybackslash}p{0.28\linewidth}@{}}
\toprule
Framework & Unknown & Data & Typical guarantee \\
\midrule
Discrete Calder\'on (resistor networks)
& Edge conductances; Ohmic law known
& Boundary voltage--current map
& Uniqueness under combinatorial conditions on the graph \\
\addlinespace
Quasilinear Calder\'on; nonlinear networks
& Spatially varying or edge-specific nonlinear law
& Nonlinear boundary voltage--current (Dirichlet-to-Neumann) map
& Uniqueness via linearization or structural assumptions; logarithmic
  stability already in the linear continuum problem \\
\addlinespace
Inverse optimization; revealed preference
& Objective parameters or function
& Decisions assumed optimal for known forward problems
& A set of consistent objectives; an estimate selected by normalization,
  hypothesis class, or regularization \\
\addlinespace
Parametric structural identifiability
& Finite parameter vector
& Noiseless input--output map
& Injectivity of the parameter-to-output map \\
\addlinespace
This paper
& One shared even generator~$\Psi$ in class~\eqref{eq:Psi-class}
& Equilibrium locations under controlled $(u,r)$, or under a boundary sweep
  at one fixed ratio
& Uniqueness up to scale; explicit inversion; fixed-ratio identification;
  local Lipschitz stability \\
\bottomrule
\end{tabular}
\end{table}

\paragraph{Inverse optimization and revealed preference.}
Inverse optimization infers an objective from decisions assumed optimal, or
nearly optimal, for a family of forward
problems~\cite{AhujaOrlin2001,ChanMahmoodZhu2025}. Established results include
the minimal perturbation of a linear objective that makes a given solution of
a linear or combinatorial program optimal~\cite{AhujaOrlin2001}, imputation of
convex objectives parametrized by pre-selected basis functions through
minimization of optimality-condition residuals~\cite{Keshavarz2011},
estimation under noisy observations with consistency
guarantees~\cite{AswaniShenSiddiq2018}, nonparametric convex risk
functions~\cite{Li2021}, parametric or kernel-based nonparametric
estimation of utility and congestion functions from observed Nash and Wardrop
equilibria through inverse variational inequalities~\cite{Bertsimas2015}, and
robust formulations for noisy and partially observed decisions that trade a
duality-gap or Karush--Kuhn--Tucker residual against a distance to the
observations~\cite{ThaiBayen2018}.
Revealed-preference theory is an early instance: finite choice data are
rationalizable by a continuous, concave, monotone utility exactly when they
satisfy a cyclical consistency condition, and the rationalizing utility is then
generally not unique~\cite{Afriat1967,Varian1982}. Common to these programs is
that observations restrict the objective to a set of consistent candidates.
Trivial objectives, for which every feasible decision is optimal, are excluded
by normalization, and a particular estimate is selected by a hypothesis class
or by regularization; when the objective is homogeneous, the imputed objective
is fixed only up to a multiplicative factor~\cite{Keshavarz2011}, the analogue
of the scale ambiguity below. The present setting is a structured
special case in which the consistent set can be computed exactly. Every
forward problem in the family $r\Psi(v)+\Psi(u-v)$ shares one unknown
function, so each equilibrium yields one homogeneous linear relation
$r\,g(v^\star)=g(u-v^\star)$ between two values of~$g$. Varying $(u,r)$
supplies every such relation, and the consistent set within
class~\eqref{eq:Psi-class} collapses to the ray $\{c\Psi:c>0\}$.
Theorem~\ref{thm:response-ident} is therefore a uniqueness and inversion
theorem for an unrestricted function rather than an estimator, and the scale
ray is the exact residual ambiguity rather than a normalization choice.
The comparison with inverse variational
inequalities~\cite{Bertsimas2015,ThaiBayen2018} is the closest. There, the
unknown is a general variational-inequality map, possibly vector-valued and
coupling many agents, sought in a chosen parametric or kernel-based hypothesis
class; the data are the equilibria that the application provides; and the
analysis concerns the computation and statistical performance of the
resulting estimate. Here the experimenter chooses the forward problems, the
unknown is one scalar generator entering every term of every objective, and
only equilibrium locations are observed. That restriction is what makes exact
recovery possible here without a hypothesis class: each designed equilibrium links~$g$ at two known
arguments, and the question becomes which designs make the generator
identifiable. Joint
variation of~$(u,r)$ suffices, and so does one fixed-ratio boundary sweep
(Theorem~\ref{thm:fixed-ratio}); equal weights do not. For the finite noisy
version of the present problem the estimators above are the natural tools.
Proposition~\ref{prop:bisection} analyzes only a target-by-target procedure
with exact controls, and nonparametric estimation from finitely many noisy
equilibria is not treated here.

\paragraph{Boundary-response and nonlinear network inverse problems.}
The Calder\'on problem recovers a conductivity from the
Dirichlet-to-Neumann map~\cite{Calderon1980,SylvesterUhlmann1987,AstalaPaivarinta2006,Uhlmann2009}.
Its discrete analogue recovers the edge conductances of a resistor network from
boundary voltage--current data, with uniqueness governed by the combinatorics
of the
graph~\cite{CurtisMorrow1990,CurtisMorrow1991,CurtisIngermanMorrow1998,CurtisMorrow2000,ColinDeVerdiere1994,GernandtRohleder2022}.
Nonlinear extensions recover quasilinear conductivities $\gamma(x,u)$ by
linearizing the Dirichlet-to-Neumann map about constant boundary data, which
reduces the problem, one level~$t$ at a time, to linear Calder\'on
problems~\cite{Sun1996,SunUhlmann1997}; higher-order linearization extends
uniqueness to broader quasilinear classes~\cite{Carstea2021}. Related work
treats boundary determination for the $p$-Laplacian~\cite{SaloZhong2012},
monotonicity-based tomography of nonlinear conductors~\cite{CorboEsposito2021},
and circular planar networks with nonlinear and signed
conductors~\cite{Johnson2012}. For the continuum Calder\'on problem, stability
under a~priori bounds is logarithmic~\cite{Alessandrini1988}, and this rate
cannot be substantially improved~\cite{Mandache2001}. Three features separate
the present problem. First, the roles of law and weights are exchanged: in
network Calder\'on problems the constitutive law is known or edge-specific and
the weights are unknown, whereas here the weights are controls and the single
shared law is the unknown. Second, the data are equilibrium locations only; no
boundary current is measured, so no Dirichlet-to-Neumann map is available.
Third, boundary control plays a role similar to the constant-data
linearization of~\cite{Sun1996}, selecting the arguments at which the
nonlinearity is probed, but the comparison is exact rather than infinitesimal:
each query compares~$g$ at the two finite arguments $u-v^\star$ and~$v^\star$.
Because the unknown is a function of one variable, probed pointwise through
a scalar implicit equation, local stability of the variable-ratio inversion is
Lipschitz (Corollary~\ref{cor:stability}). This does not improve on Calder\'on
stability. The logarithmic rate there concerns a coefficient of several
spatial variables recovered from a boundary operator, a problem with
different data and unknowns, and the contrast only indicates that the present
problem does not inherit that source of ill-posedness. The fixed-ratio
iteration of Theorem~\ref{thm:fixed-ratio} is not covered by this estimate.
What those programs establish and the present paper does not is recovery on
general multi-vertex networks, where stationarity is a coupled system.
Theorem~\ref{thm:chain} treats one coupled configuration, a chain with two
free vertices; general graphs with several free vertices remain open.

\paragraph{Structural identifiability and related inverse questions.}
Structural identifiability asks whether model parameters are determined by
noiseless, complete input--output
data~\cite{BellmanAstrom1970,LjungGlad1994,WalterPronzato1997}; methods for
systems-biology models are compared in~\cite{Chis2011,Villaverde2016}, and
practical identifiability from finite noisy data is assessed separately, for
example by profile likelihood~\cite{Raue2009}. These frameworks are
predominantly parametric. Theorem~\ref{thm:response-ident} is a nonparametric
structural statement: the unknown ranges over an infinite-dimensional class,
the residual non-identifiability is a one-parameter scaling group identified
exactly, and the proof yields an explicit inverse. Section~\ref{sec:numerics}
supplies the parametric practical counterpart within~$J_p$. Equation
discovery selects governing terms from a library given trajectory
data~\cite{BruntonProctorKutz2016}; it is a model-selection procedure rather
than a uniqueness result. The inverse problem of the calculus of variations
asks whether a given system of differential equations is the Euler--Lagrange
system of some Lagrangian; Douglas classified when a second-order system in
two dependent variables admits one, and how many~\cite{Douglas1941}. Here the
variational structure is given and only its minimizers are observed. Finally,
a lower semicontinuous proper convex function on a Banach space is determined
up to an additive constant by its subdifferential~\cite{Rockafellar1966}.
Equilibrium locations carry less information than gradients, since they are
invariant under positive rescaling of the objective; this is why the residual
ambiguity in Theorem~\ref{thm:response-ident} is multiplicative, the additive
constant being fixed by $\Psi(0)=0$. The inversion
formula~\eqref{eq:inversion} can be read as recovering ratios of gradients
from minimizer locations.

\paragraph{Functional equations and divergences.}
Theorem~\ref{thm:fixed-ratio} is an instance of Schr\"oder's functional
equation $g\circ T=\varrho^{-1}g$, with a constant $\varrho>1$, and of Koenigs'
linearization of a map near an attracting fixed
point~\cite{Schroeder1870,Koenigs1884,KuczmaChoczewskiGer1990}. That theory is
classical. It concerns existence and uniqueness of solutions for a given~$T$;
the observation used here is that a fixed-ratio response curve supplies~$T$,
with the unknown generator as a solution. The same machinery has been used in
an inverse problem before, to recover the vector field of a scalar autonomous
ordinary differential equation from sampled trajectories by estimating the
time-one map and solving Schr\"oder's and Julia's equations, with continuous
dependence of the recovered field on that map~\cite{AlfaroVigo2020}. There the conjugated map is a flow map estimated from time
series; here it is assembled from the two complementary equilibrium locations
of a static fixed-ratio sweep, and the unknown is a constitutive generator
rather than a vector field.

Closer in form, Peluso and Trannoy start from an observed nonlinear sharing
curve at a prescribed weight, form its complementary map, and use iterates to
construct compatible marginal utilities~\cite{PelusoTrannoy2012}. Their result
is a rationalizability construction with the weight fixed in advance, not
joint identification of an unknown weight and a common generator.

The response also belongs to the classical theory of deviation means. With
$E_g(a,v):=g(a-v)$, which is a deviation and hence a quasideviation, $V(u)$ is
the two-point weighted quasideviation mean of $(0,u)$ with weights $(r,1)$.
Translation equivariance gives $M_{g,r}(a,b)=a+V(b-a)$ for $a\le b$, while
oddness determines the reversed order. Deviation and weighted quasideviation
means have developed characterization and equality
theories~\cite{Pales1982,DaroczyPales1982,PalesPasteczka2020}; fixed-weight
binary equality, including unknown nonsymmetric weights under additional
differentiability assumptions, is also known in narrower quasi-arithmetic and
Bajraktarevi\'c classes~\cite{PalesZakaria2020,PalesZakaria2025}. We have not
located in that literature the specific inverse statement proved here for a
general odd translation-invariant deviation generator: from $V$ at one fixed
unequal but initially unknown ratio, one identifies both the ratio and the
generator up to a positive scale. The failure at $r=1$ in
Theorem~\ref{thm:fixed-ratio} shows why the unequal-weight qualification is
essential.

Algebraically, $2J(x/y)$ is the additive symmetrization of the Itakura--Saito
(COSH) divergence~\cite{Eguchi2006,CichockiAmari2010,BanerjeeMBG2005,Amari2016},
and the ratio-composition characterization of~$J$ reduces to d'Alembert's
functional equation~\cite{Aczel1966,WashburnAczel}. These results characterize
a candidate cost; they do not address recovery of a cost from equilibrium
data, and none of the inverse results depends on them. Surveys of inverse
problems and regularization~\cite{EnglHankeNeubauer1996,Isakov2006,Tarantola2005}
supply general methodological context.

\paragraph{Summary of novelty.}
Established in the literature are uniqueness of spatial coefficients and edge
laws from boundary current--voltage maps, with logarithmic stability in the
continuum; estimation of objectives and of variational-inequality functions
from observed decisions and equilibria, with guarantees that depend on the
observations and the hypothesis class; parametric structural identifiability;
prescribed-weight sharing-rule constructions; fixed-weight identification in
narrower classes of means; and the classical Schr\"oder--Koenigs theory,
including its use to recover a vector field from sampled trajectories. Against
this background, the new elements are those of
Section~\ref{sec:contributions}: identification of an unrestricted shared
generator from location-only data, including simultaneous recovery of the
ratio and generator from one fixed unequal-ratio sweep, together with
finite-query bounds and two graph extensions. Open problems are listed in
Section~\ref{sec:discussion}.

\subsection{Organization}
Section~\ref{sec:forward} records the forward problem; Section~\ref{sec:ident}
proves the identification results, the finite-target and stability bounds,
and the graph extensions; Section~\ref{sec:reciprocal}
introduces the reciprocal candidate~$J$, the family~$J_p$, and the specialized
discriminators; Section~\ref{sec:numerics} reports finite-grid experiments;
Section~\ref{sec:outlook} discusses physical interpretation and scope;
Section~\ref{sec:discussion} concludes. Gradient-flow and Langevin realizations
of the forward equilibrium are standard once the potential is fixed and are
collected in Supplemental~S2--S3.

\FloatBarrier

\section{Forward problem on weighted graphs}
\label{sec:forward}
\label{sec:cost}

Throughout, $x>0$ denotes a multiplicative deviation from a site-specific
reference. Write $F(x)$ for a reciprocal scalar cost on~$x$ and
$\Psi(t):=F(e^t)$ for its log-coordinate representation. The inverse
analysis in Section~\ref{sec:ident} applies to the general class
\begin{equation}
\label{eq:Psi-class}
\begin{aligned}
&\Psi\in C^2(\mathbb{R}),\quad
\Psi\text{ even},\quad
\Psi(0)=0,\quad
\Psi''(t)>0\ \text{for all }t,\\
&g:=\Psi'\quad\text{(odd, strictly increasing, }g(0)=0\text{)}.
\end{aligned}
\end{equation}
In particular, $g'(0)=\Psi''(0)>0$. This nondegeneracy at the origin is
essential to the unknown-ratio conclusion of
Theorem~\ref{thm:fixed-ratio}.
The reciprocal cost $J$ of Section~\ref{sec:reciprocal} is the special case
$g=\sinh$; it is not assumed in the reconstruction theorem.

\subsection{Graph functional}
Fix a finite set of sites~$X$ and a symmetric edge set~$\mathcal{E}$ of
unordered adjacent pairs. Equip each site~$s$ with a weight~$w_s>0$ and each
edge~$\{s,s'\}$ with a weight~$w_{ss'}>0$. A positive multiplier field is a map
$m:X\to\mathbb{R}_{>0}$; a boundary condition pins a subset
$X_\partial\subseteq X$ to prescribed positive values~$b$. The uniform field
assigns~$1$ to every site and is the unique absolute zero of the functional
constructed below. Collectively, the graph, weights, references, and
measurement transforms constitute the representation~$\mathcal{M}$.

Let $a_s>0$ denote an unnormalized site attribute and $\kappa_{ss'}>0$ the
corresponding edge attribute. Because site and edge attributes generally have
different physical dimensions, their relative strength is controlled by an
explicit site-to-edge coupling~$\lambda>0$. Let $a_{\mathrm{ref}}$ and
$\kappa_{\mathrm{ref}}$ be positive reference scales fixed once for the entire
protocol, and write $\bar a_s=a_s/a_{\mathrm{ref}}$ and
$\bar\kappa_{ss'}=\kappa_{ss'}/\kappa_{\mathrm{ref}}$. The reference scales must
not be renormalized separately in each sweep condition. Then
\[
w_s=\bar a_s,\qquad w_{ss'}=\lambda\,\bar\kappa_{ss'}.
\]
Physical quantities that could motivate these attributes include normalized
compartment volume or molecule-count capacity for sites, and contact area,
channel density, or measured transport conductance for
edges~\cite{Weiss1996}. A common overall multiplicative scale of the cost leaves
every constrained minimizer unchanged and is therefore not identifiable from
equilibrium locations; Supplemental~S4 records the resulting
non-identifiabilities. Two objects are kept distinct from these graph weights:
statistical inverse-variance weights of the observation model, and any global
scale relating the functional to a physical energy
(Section~\ref{sec:outlook}).

\begin{definition}[Reference-relative deviation functional]
\label{def:deviation-functional}
For $\Psi$ in class~\eqref{eq:Psi-class} and a positive multiplier field~$m$
on~$X$,
\begin{equation}
\label{eq:graph-functional}
C_{\mathcal{M},\Psi}[m]
  \;=\; \sum_{s\in X} w_s\,\Psi\bigl(\log m(s)\bigr)
    \;+\; \sum_{\{s,s'\}\in\mathcal{E}}
      w_{ss'}\,\Psi\!\left(\log\frac{m(s)}{m(s')}\right).
\end{equation}
When $\Psi(t)=\cosh t-1$ we write $C_{\mathcal{M}}$ for
$C_{\mathcal{M},\Psi}$ and recover the reciprocal specialization of
Section~\ref{sec:reciprocal}. Equivalently, in free log-coordinates~$y$,
\begin{equation}
\label{eq:graph-functional-log}
C_{\mathcal{M},\Psi}(y)
  \;=\; \sum_{s} w_s\,\Psi(y_s)
    \;+\; \sum_{\{s,s'\}} w_{ss'}\,\Psi(y_s-y_{s'}),
\end{equation}
with pinned coordinates substituted as boundary values.
\end{definition}

The edge term is well defined because $\Psi$ is even. Site terms penalize
absolute log-deviation from the reference; edge terms penalize relative
log-deviation between adjacent sites. Throughout, $C_{\mathcal{M},\Psi}$ is a
dimensionless bookkeeping functional rather than a thermodynamic free energy
(Section~\ref{sec:outlook}).

\medskip
\noindent\textbf{Pinned-site constants and the reduced objective.}
For a boundary assignment~$b$ on~$X_\partial$, write $m_b(y)$ for the full
multiplier field that equals $e^{y_s}$ on free sites and $b(s)$ on pinned sites.
The sum in~\eqref{eq:graph-functional} includes pinned sites, which contribute
constants. The boundary-reduced objective on free log-coordinates is
\[
C^{\mathrm{free}}_{\mathcal{M},\Psi,b}[y]
  :=C_{\mathcal{M},\Psi}[m_b(y)]
    -\sum_{s\in X_\partial}w_s\,\Psi\bigl(\log b(s)\bigr).
\]
At fixed~$b$ we abbreviate $\widetilde C:=C^{\mathrm{free}}_{\mathcal{M},\Psi,b}$.
Pinned constants cancel in free-coordinate gradients and in cost differences at
fixed boundary, so they affect no minimizer. Locality, pairwise additivity, and
the absence of higher-order interactions are modeling assumptions. The termwise
propagation of the ratio-composition-law (RCL) characterization of~$J$ under a pairwise homogeneous
ansatz is recorded as Corollary~\ref{thm:graph-propagation} in
Section~\ref{sec:reciprocal}.

\begin{proposition}[Static minimization]
\label{prop:P3}
Let $\Psi$ be in class~\eqref{eq:Psi-class}, let $\mathcal{M}$ have weights
$w_s,w_{ss'}>0$, and let $(X_\partial,b)$ be a boundary condition with nonempty
free set. Then:
\begin{enumerate}
  \item $C_{\mathcal{M},\Psi}[m]\ge0$ for every field~$m$;
  \item over unrestricted fields, the uniform field is the unique absolute
    minimizer and has $C_{\mathcal{M},\Psi}=0$;
  \item for every fixed boundary assignment there is a unique constrained
    minimizer on the feasible set;
  \item if at least one pinned value differs from one, the constrained minimum
    of $C_{\mathcal{M},\Psi}$ is strictly positive.
\end{enumerate}
\end{proposition}

\begin{proof}
Every summand is a nonnegative multiple of~$\Psi$ (even, vanishing only at~$0$)
composed with a log ratio, which gives (i), (ii), and~(iv). For~(iii), the
free-site Hessian contributions are $\operatorname{diag}(w_s\Psi''(y_s))\succ0$
and the edge terms are convex compositions with affine maps, so
$\widetilde C$ is strictly convex on every affine feasible set.

Coercivity holds for every generator in class~\eqref{eq:Psi-class}, not only
for the reciprocal specialization. Fix any $t_0>0$ and set $c:=g(t_0)>0$, which
is positive because $g$ is strictly increasing with $g(0)=0$. Convexity
of~$\Psi$ gives $\Psi(t)\ge\Psi(t_0)+c\,(t-t_0)\ge c\,t-c\,t_0$ for $t\ge t_0$,
and evenness extends this to
\[
\Psi(t)\ \ge\ c\,|t|-c\,t_0\qquad(t\in\mathbb{R}),
\]
since for $|t|\le t_0$ the right-hand side is nonpositive while $\Psi\ge0$. So
$g$ keeps a sign bounded away from zero outside any neighborhood of the origin
and $\Psi$ grows at least linearly. Discarding the nonnegative edge terms and
writing $w_{\min}:=\min_s w_s>0$ over the free sites,
\[
\widetilde C(y)\ \ge\ w_{\min}\,c\sum_{s\ \mathrm{free}}|y_s|
  \;-\;w_{\min}\,c\,t_0\,\bigl|\{s\ \mathrm{free}\}\bigr|
\ \longrightarrow\ \infty
\qquad\text{as }\|y\|\to\infty,
\]
so sublevel sets are bounded, a minimizer exists, and strict convexity makes it
unique. Supplemental~S2 records the same argument in the reciprocal
specialization.
\end{proof}

The constrained minimizer is the observable for the inverse problem:
Section~\ref{sec:ident} recovers the generator from the response map, and
Section~\ref{sec:numerics} quantifies finite-grid discrimination. Equilibria are
realized by the associated convex gradient flow; details are in Supplemental~S2--S3.

\section{Uniqueness and reconstruction from equilibrium maps}
\label{sec:ident}

We work throughout this section with a generator $\Psi$ in the class
\eqref{eq:Psi-class}, writing $g=\Psi'$. The two-site response operator
\eqref{eq:response-op} is well defined by strict convexity.
Figure~\ref{fig:reconstruction-query} summarizes the basic reconstruction query.

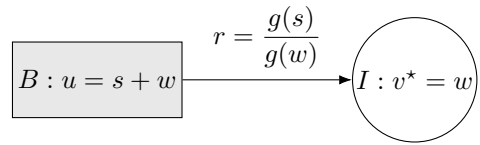
\begin{figure}[htbp]
\centering
\begin{tikzpicture}[
  free/.style={circle,draw=black,fill=white,minimum size=11mm,inner sep=1pt},
  pinned/.style={rectangle,draw=black,fill=gray!18,minimum size=10mm,inner sep=2pt},
  every node/.style={font=\footnotesize}
]
  \node[pinned] (B) at (0,0) {$B:u=s+w$};
  \node[free] (I) at (4.2,0) {$I:v^\star=w$};
  \draw[-{Latex}] (B) -- node[above] {$r=\dfrac{g(s)}{g(w)}$} (I);
  \node[align=center,anchor=north,font=\footnotesize,
        inner sep=0pt,outer sep=0pt] at ($(B.south)!0.5!(I.south)+(0,-3.5mm)$)
    {site weight $\alpha$, edge weight $\beta$, $r=\alpha/\beta$\\[2.5ex]
     tune $r$ until $v^\star=w$; attained $r=g(s)/g(w)$.};
\end{tikzpicture}
\caption{Reconstruction query for Theorem~\ref{thm:response-ident}. With boundary
$u=s+w$ and target equilibrium~$w$, the weight ratio that realizes
$v^\star=w$ equals $g(s)/g(w)$.}
\label{fig:reconstruction-query}
\end{figure}

\subsection{Main reconstruction theorem}
\label{sec:main-recon}

\begin{theorem}[Uniqueness and reconstruction from the equilibrium response map]
\label{thm:response-ident}
\label{prop:global-ident}
Consider a free site~$I$ (log-coordinate~$v$, weight~$\alpha>0$) joined by a
single edge (weight~$\beta>0$) to a neighbor pinned at~$u>0$, and set
$r:=\alpha/\beta$. For $\Psi$ in class~\eqref{eq:Psi-class}, the objective
$r\Psi(v)+\Psi(u-v)$ is strictly convex, with unique minimizer
$v^\star=\mathcal{R}_{\Psi}(u,r)\in(0,u)$ characterized by
\begin{equation}
\label{eq:gadget-stat}
r\,g(v^\star)=g(u-v^\star),
\end{equation}
equivalently the inversion formula
\begin{equation}
\label{eq:inversion}
r=\frac{g(u-v^\star)}{g(v^\star)}.
\end{equation}
As $(u,r)$ ranges over $(0,\infty)^2$, the complete response map
$\mathcal{R}_{\Psi}$ determines $g(s)/g(w)$ for every $s,w>0$. Without a
scale calibration this determines~$g$ only up to a positive multiplicative
constant. Within the calibrated class $g'(0)=1$, it determines~$g$ uniquely,
and hence determines~$\Psi$ via $\Psi(t)=\int_0^t g$. In particular, no
calibrated generator distinct from a given~$\Psi$ can reproduce
$\mathcal{R}_{\Psi}$ on all $(u,r)$. The statement is structural and noiseless:
it presumes exact access to the complete map. Finite-data discrimination is
treated parametrically in Section~\ref{sec:numerics}.
\end{theorem}

\begin{proof}
The hypothesis $\Psi''(t)>0$ implies strict convexity and $g'(t)>0$, so a
unique minimizer solves~\eqref{eq:gadget-stat}. Fix $u>0$. Since $g$ is odd and
strictly increasing with $g(0)=0$, the objective derivative
$r g(v)-g(u-v)$ is strictly negative on $v\le0$ and strictly positive on
$v\ge u$, so $v^\star\in(0,u)$. Write $G(v,r)=r\,g(v)-g(u-v)$. Then
$\partial_v G=r\,g'(v)+g'(u-v)>0$ and $\partial_r G=g(v)$, so the implicit
function theorem gives
\[
\frac{dv^\star}{dr}
  =-\frac{g(v^\star)}{r\,g'(v^\star)+g'(u-v^\star)}<0
\]
on~$(0,u)$. Thus $v^\star$ decreases strictly in~$r$, with endpoint limits
$v^\star\to u$ as $r\to0^+$ and $v^\star\to0$ as $r\to\infty$. Hence $v^\star$
sweeps~$(0,u)$ bijectively. For any target $w\in(0,u)$ a unique ratio places
the minimizer at $v^\star=w$, and~\eqref{eq:inversion} reads
$r=g(u-w)/g(w)$. Given arbitrary $s,w>0$, choose $u=s+w$: the ratio that places
the minimizer at~$w$ is $g(s)/g(w)$. The response map therefore yields every
positive ratio of~$g$. Fix $w_0>0$ and set $h(s):=g(s)/g(w_0)$ for $s>0$, which
the response map supplies on the open half-line only; extend it to the endpoint
by the convention
\[
h(0):=0=\lim_{s\to0^+}\frac{g(s)}{g(w_0)},
\]
forced by continuity of~$g$ and $g(0)=0$. Then $h$ is known and $C^1$
on~$(0,\infty)$ and continuous at~$0$, and without calibration it
determines~$g$ only up to the unknown factor~$g(w_0)$. Read the slope
at the endpoint as the right derivative
\[
h'(0^+):=\lim_{s\to0^+}\frac{h(s)-h(0)}{s}
        =\lim_{s\to0^+}\frac{g(s)}{s\,g(w_0)}
        =\frac{g'(0)}{g(w_0)},
\]
which is exactly the near-origin information carried by the positive-side
response map. Under the calibration $g'(0)=1$ this gives $h'(0^+)=1/g(w_0)$,
hence $g(s)=h(s)/h'(0^+)$ for $s>0$. The odd extension fixes~$g$ on
$(-\infty,0)$, and $\Psi(t)=\int_0^t g$ recovers the generator.
\end{proof}

Equation~\eqref{eq:inversion} is not merely an injectivity certificate.
The reconstruction mechanism is unusually direct: boundary control sets
the two arguments at which~$g$ is compared, while the weight ratio that
places the observed equilibrium at a prescribed location reports their
derivative ratio. The equilibrium location is not merely evidence that
two candidate generators differ; across the complete response map it
provides the information needed to reconstruct the generator itself,
apart from the unavoidable global scale.

\subsection{Identification at a single fixed weight ratio}
\label{sec:fixed-ratio}
Theorem~\ref{thm:response-ident} varies both controls. A boundary sweep at a
single fixed ratio already suffices, and the ratio need not be known. With one
ratio, each equilibrium links~$g$ at two points, and these links form chains
that accumulate at the origin; the regularity of~$g$ there ties the chains
together.

\begin{theorem}[Identification from one boundary sweep at a fixed, possibly unknown, ratio]
\label{thm:fixed-ratio}
Let $\Psi$ be in class~\eqref{eq:Psi-class}, fix $r>0$, and suppose the
boundary response $V(u):=\mathcal{R}_{\Psi}(u,r)$ is known for
$u\in(0,\bar u]$, where $0<\bar u\le\infty$ (read $(0,\infty)$ when $\bar u=\infty$). The
ratio~$r$ is constant over the sweep but need not be known.
\begin{enumerate}
  \item[(i)] The ratio is identified:
    $r=\lim_{u\to0^+}u/V(u)-1$.
  \item[(ii)] If $r=1$, then $V(u)=u/2$ for every generator, and the curve
    carries no information about~$\Psi$.
  \item[(iii)] Let $r\neq1$ and $\varrho:=\max\{r,1/r\}>1$. Write $\ell(u)$
    for the larger and $\tau(u)$ for the smaller of $V(u)$ and $u-V(u)$; thus
    $\ell(u)=u-V(u)$ if $r>1$ and $\ell(u)=V(u)$ if $r<1$, and which case holds
    is read from~$V$, since $V(u)<u/2$ exactly when $r>1$. Then $\ell$ is a
    continuous, strictly increasing bijection of $(0,\bar u]$ onto
    $(0,\bar s]$, where $\bar s:=\ell(\bar u)$ ($\bar s=\infty$ if
    $\bar u=\infty$), and $T:=\tau\circ\ell^{-1}$ is continuous and strictly
    increasing, with
    \[
    0<T(s)<s,\qquad g\bigl(T(s)\bigr)=\varrho^{-1}g(s)\qquad(0<s\le\bar s).
    \]
\end{enumerate}
In case~(iii), for every $s\in(0,\bar s]$, with $T^{\circ n}$ the $n$-fold
composite,
\begin{equation}
\label{eq:koenigs}
g(s)=g'(0)\lim_{n\to\infty}\varrho^{n}\,T^{\circ n}(s).
\end{equation}
Consequently, for $r\neq1$ the curve~$V$ on $(0,\bar u]$ determines~$r$ and
determines~$g$ on $(0,\bar s]$ up to the positive factor~$g'(0)$, uniquely in the
calibrated class $g'(0)=1$. When $\bar u=\infty$ it determines~$\Psi$ up to a
positive scale, uniquely after calibration.
\end{theorem}

\begin{proof}
By Theorem~\ref{thm:response-ident}, $V(u)\in(0,u)$ and
$r\,g(V(u))=g(u-V(u))$. For~(i), both $V(u)$ and $u-V(u)$ lie in $(0,u)$, and
$g(t)/t\to g'(0)>0$ as $t\to0^+$, so
\[
\frac{u-V(u)}{V(u)}
=r\,\frac{g(V(u))/V(u)}{g(u-V(u))/(u-V(u))}\longrightarrow r
\qquad(u\to0^+).
\]
Statement~(ii) is the equal-weight symmetry of
Section~\ref{sec:symmetric-design}. For~(iii), let
$T_\varrho(s):=g^{-1}\bigl(g(s)/\varrho\bigr)$ for $s>0$, which is well
defined because $g(s)/\varrho\in(0,g(s))$ lies in the range of~$g$. The map
$T_\varrho$ is continuous and strictly increasing, with $0<T_\varrho(s)<s$.
If $r>1$, then $g(u-V)=r\,g(V)>g(V)$, so $V(u)<u/2$; if $r<1$, then
$V(u)>u/2$. In both cases the stationarity relation reads
$\tau(u)=T_\varrho(\ell(u))$, so $u=\ell(u)+T_\varrho(\ell(u))$. The map
$s\mapsto s+T_\varrho(s)$ is continuous and strictly increasing, tends to~$0$
as $s\to0^+$, and is at least~$s$. Hence $\ell$ is its inverse
on~$(0,\bar u]$, a continuous increasing bijection onto $(0,\bar s]$, and
$T=\tau\circ\ell^{-1}=T_\varrho$ on $(0,\bar s]$. The orbit $T^{\circ n}(s)$
decreases and stays positive, so it converges to some $t_\infty\ge0$; if
$t_\infty>0$, continuity would give $T(t_\infty)=t_\infty$, contradicting
$T(t_\infty)<t_\infty$. Thus $T^{\circ n}(s)\to0$. Induction gives $g(T^{\circ n}(s))=\varrho^{-n}g(s)$,
and therefore
\[
\varrho^{n}\,T^{\circ n}(s)
=g(s)\,\frac{T^{\circ n}(s)}{g\bigl(T^{\circ n}(s)\bigr)}
\longrightarrow\frac{g(s)}{g'(0)},
\]
which is~\eqref{eq:koenigs}. Because $\ell$ and~$T$ are computed from~$V$
alone, two generator--ratio pairs with the same curve have the same ratio,
by~(i), and the same values of $g/g'(0)$ on~$(0,\bar s]$. With $\bar u=\infty$ the odd
extension and $\Psi(t)=\int_0^t g$ complete the proof.
\end{proof}

The relation $g\circ T=\varrho^{-1}g$ is Schr\"oder's functional
equation~\cite{Schroeder1870} for the map~$T$. Extended by $T(0)=0$, $T$ is
$C^1$ on $[0,\bar s]$ with $T'(0)=1/\varrho$, because
$T'(s)=g'(s)/\bigl(\varrho\,g'(T(s))\bigr)\to1/\varrho$ as $s\to0^+$; so~$0$
is an attracting fixed point with multiplier~$1/\varrho$,
and~\eqref{eq:koenigs} is Koenigs' linearization
limit~\cite{Koenigs1884,KuczmaChoczewskiGer1990}. The classical theory gives
conditions under which such solutions exist and are unique. Here a solution,
$g$ itself, is given, and the exact identity in the proof forces the displayed
limit; convergence is not inferred from $C^1$ regularity of an arbitrary map.
Prior sharing-rule work constructs compatible utilities from a response curve
at a prescribed weight~\cite{PelusoTrannoy2012}. To our knowledge, the new
point here is the joint identification: within class~\eqref{eq:Psi-class}, one
unequal-ratio sweep determines both the initially unknown ratio and the
generator up to scale. The novelty claim does not concern the iteration
machinery.

The value of~$\varrho$ is not needed for the reconstruction up to scale:
since $g(T^{\circ n}(s))/g(T^{\circ n}(s_0))=g(s)/g(s_0)$ for every~$n$,
\[
\frac{g(s)}{g(s_0)}=\lim_{n\to\infty}\frac{T^{\circ n}(s)}{T^{\circ n}(s_0)}
\qquad(s,s_0\in(0,\bar s]),
\]
and $r$ then follows from $r=g(u-V(u))/g(V(u))$ at any~$u$.

Theorem~\ref{thm:fixed-ratio} trades the ability to vary the weight ratio for
an iterative reconstruction that probes progressively smaller arguments.
Although the exact reconstruction formula is convergent, its numerical
conditioning under perturbed response curves remains unresolved. In
particular, the factor~$\varrho^{n}$ in~\eqref{eq:koenigs} amplifies
perturbations in the terminal iterate, errors in the response-derived map~$T$
can accumulate throughout the iteration, and stopping the iteration early
leaves a linearization bias. The theorem establishes exact-data
identification, not a uniform stability estimate or a noise-robust
reconstruction algorithm; Supplemental~S12 checks it on noise-free responses.
Theorem~\ref{thm:response-ident},
by contrast, gives each ratio of~$g$ from a single equilibrium without a
limit. This pointwise form is what Corollaries~\ref{cor:finite-query}
and~\ref{cor:stability} and Proposition~\ref{prop:bisection} make
quantitative.

\subsection{Finite-target recovery and stability}
\label{sec:finite-target}

\begin{corollary}[Finite-target ratio reconstruction]
\label{cor:finite-query}
Fix a reference $w_0>0$ and targets $s_1,\ldots,s_n>0$, and set $u_i=s_i+w_0$.
Under an inverse-response (tune-to-target) oracle that can adjust $r_i$
until $\mathcal{R}_{\Psi}(u_i,r_i)=w_0$ exactly, the attained control equals
\[
r_i=\frac{g(s_i)}{g(w_0)},
\]
so $n$ targets determine~$g$ on the selected grid relative to the common
scale~$g(w_0)$. A fixed finite target list determines these ratios but not
the generator between target points. Absolute scale additionally requires the near-origin slope
$h'(0^+)=1/g(w_0)$ of Theorem~\ref{thm:response-ident}, so unit-curvature
calibration is available only if the protocol provides a source for it: either
a sequence of targets $s_k\to0^+$, or an equivalent one-sided slope
estimator for $s\mapsto g(s)/g(w_0)$ at the origin. Absent such a declaration
the recovered object is the ratio function~$h$, not~$g$ itself.

Under ordinary forward observation one instead evaluates
$r\mapsto\mathcal{R}_{\Psi}(u_i,r)$ (or solves the scalar stationarity equation)
and must locate $r_i$ by root-finding. In that model a single exact query per
target is not assumed: if a search returns $\hat r_i$ with residual
$\varepsilon_v:=|\mathcal{R}_{\Psi}(u_i,\hat r_i)-w_0|$ and $\hat r_i$ lies in
the neighborhood~$U$ of Corollary~\ref{cor:stability}, then
\[
\bigl|\log(\hat r_i/r_i)\bigr|
\le\frac{\varepsilon_v}{c_0\inf_U r}
\]
with the local constants of that corollary. A practical stopping rule is to
terminate when $\varepsilon_v\le\varepsilon$ for a prescribed tolerance
$\varepsilon>0$, which yields an explicit log-ratio accuracy bound.
\end{corollary}

Exact reconstruction assumes exact control and equilibrium localization.
In any finite-precision setting, the relevant question is therefore how
errors in the observed target location propagate into the recovered ratio
$g(s)/g(w)$. Corollary~\ref{cor:stability} gives a local answer: away from
the origin and for bounded curvature, this inversion is Lipschitz stable.

\begin{corollary}[Local Lipschitz stability of ratio recovery]
\label{cor:stability}
Let $g=\Psi'$ for a generator~$\Psi$ in class~\eqref{eq:Psi-class}, and
suppose that on a compact interval $K=[a,b]\subset(0,\infty)$ one has
$0<m\le g'(t)\le M$. Define
\[
g_-:=\min_{t\in K}g(t)>0,\qquad g_+:=\max_{t\in K}g(t).
\]
(The positivity of~$g_-$ follows from $g(0)=0$, $g'>0$ on~$(0,\infty)$, and
compactness of~$K$.) Fix $s,w\in\operatorname{int}K=(a,b)$, set $u=s+w$ and
$r_{*}:=g(s)/g(w)$, so that $v_{*}=w$. With $F(v,r):=r\,g(v)-g(u-v)$, the
implicit-function theorem yields a
$C^{1}$ branch $r\mapsto v(r)$ near $r_{*}$ with
\[
v'(r)=-\frac{g(v(r))}{r\,g'(v(r))+g'(u-v(r))}.
\]
Let $U=(r_1,r_2)$ be a bounded open interval with $r_{*}\in U$ and closure
$\overline U\subset(0,\infty)$, chosen small enough that the branch is defined
with $v(r)\in K$ and $u-v(r)\in K$ for every $r\in\overline U$; such a~$U$
exists by continuity of $r\mapsto v(r)$ because $s,w\in\operatorname{int}K$. In particular $\inf_U r=r_1>0$ and $\sup_U r=r_2<\infty$, and since
$U$ is an interval the whole segment between any $\hat r\in U$ and~$r_{*}$
stays in~$U$. Then for every $r\in U$,
\[
c_0|r-r_{*}|\le |v(r)-w|\le C_0|r-r_{*}|,
\]
where
\[
c_0:=\frac{g_-}{M(1+\sup_U r)},\qquad
C_0:=\frac{g_+}{m(1+\inf_U r)}.
\]
Consequently, whenever $\hat r\in U$, an equilibrium residual
$\varepsilon_v:=|v(\hat r)-w|$ implies the local bound
$|\log(\hat r/r_{*})|\le\varepsilon_v/(c_0\inf_U r)$; the implication is not
claimed for ratios outside~$U$. Finally, the log-ratio map
$\Theta(u,v):=\log\bigl(g(u-v)/g(v)\bigr)$, which returns the logarithm of the
ratio that places the equilibrium at~$v$ under boundary~$u$, is Lipschitz on
$\{(u,v):v\in K,\ u-v\in K\}$, with
\[
|\Theta(u_1,v_1)-\Theta(u_2,v_2)|
\le\frac{M}{g_-}\bigl(|u_1-u_2|+2|v_1-v_2|\bigr).
\]
\end{corollary}

\begin{proof}
The inversion formula~\eqref{eq:inversion} gives $r_*=g(s)/g(w)$. Since $F(w,r_*)=0$ and
$\partial_v F(w,r_*)=r_*g'(w)+g'(s)\ge m(r_*+1)>0$, the implicit-function
theorem supplies the $C^1$ branch and the formula for $v'(r)$. On~$U$ one has
$g_-\le g(v(r))\le g_+$ and
$m(1+\inf_U r)\le r g'(v(r))+g'(u-v(r))\le M(1+\sup_U r)$, hence
$c_0\le|v'(r)|\le C_0$ on~$U$. Because $U$ is an interval, the mean-value
theorem applies on the segment joining~$r$ to~$r_{*}$ and yields the two-sided
estimate; the same segment and $\inf_U r=r_1>0$ give
$|\log x-\log y|\le|x-y|/r_1$ for $x,y\in U$, hence the logarithmic bound. Differentiating~$\Theta$ yields
$|\partial_u\Theta|\le M/g_-$ and $|\partial_v\Theta|\le 2M/g_-$ on the
indicated set, which is convex, so the mean-value theorem implies the
Lipschitz estimate.
\end{proof}

The last estimate is the errors-in-variables form of the reconstruction. In
the protocol of Theorem~\ref{thm:response-ident} the intended value is
$\Theta(u,w)=\log\bigl(g(s)/g(w)\bigr)$. If the boundary is realized as
$u+\delta u$ and the equilibrium is located at $w+\delta v$, the attained
ratio is $\exp\Theta(u+\delta u,w+\delta v)$, so, for errors small enough
that $w+\delta v$ and $u+\delta u-w-\delta v$ remain in~$K$, the recovered
log-ratio is in error by at most $(M/g_-)(|\delta u|+2|\delta v|)$. Boundary errors and
location errors therefore enter at the same Lipschitz order; the
errors-in-variables experiments of Section~\ref{sec:numerics} illustrate the
consequence in a parametric setting. The constants degenerate as~$K$
approaches the origin, where $g_-\to0$. That is where the calibration slope
$h'(0^+)$ of Theorem~\ref{thm:response-ident} is read, so
Corollary~\ref{cor:stability} controls ratios away from the origin but not the
absolute scale.

Corollary~\ref{cor:finite-query} bounds the error of a search that has
already attained a small residual, but it does not analyze a search on noisy
observations. The next proposition does so for bisection.

\begin{proposition}[Noisy finite-target recovery by bisection]
\label{prop:bisection}
Fix one target of Corollary~\ref{cor:finite-query}, $u=s+w_0$ with
$s,w_0>0$, and let $r_*=g(s)/g(w_0)$. Suppose a bracket
$[r_L,r_U]\subset(0,\infty)$ containing~$r_*$ is known, and that
$|\partial_r\mathcal{R}_{\Psi}(u,r)|\ge c_0>0$ for all $r\in[r_L,r_U]$ (for
instance with the constant~$c_0$ of Corollary~\ref{cor:stability} when
$[r_L,r_U]\subset U$). The controls $u$ and~$r$ are realized exactly, and a
query at ratio~$r$ returns $\tilde v=\mathcal{R}_{\Psi}(u,r)+e$. Fix a
threshold $\eta>0$ and a query budget~$N$. Starting from $[a,b]=[r_L,r_U]$,
query the midpoint~$\mu$; set $a:=\mu$ if $\tilde v-w_0>\eta$, set $b:=\mu$ if
$\tilde v-w_0<-\eta$, and otherwise stop and return $\hat r:=\mu$; repeat on
the updated bracket. If no stop occurs within $N$ queries, return the
midpoint~$\hat r$ of the final bracket.
\begin{enumerate}
  \item[(i)] If every query error satisfies $|e|\le\eta$, then
  \[
  \bigl|\log(\hat r/r_*)\bigr|
  \le\frac{1}{r_L}\max\Bigl\{\frac{2\eta}{c_0},\
  \frac{r_U-r_L}{2^{N+1}}\Bigr\}.
  \]
  \item[(ii)] If each query returns the mean of $n_{\mathrm{rep}}$ fresh
  replicates whose errors are independent $\mathcal{N}(0,\sigma^2)$ variables,
  independent of the query history, then with
  \[
  \eta=\frac{\sigma}{\sqrt{n_{\mathrm{rep}}}}\sqrt{2\log\frac{2N}{q}},
  \qquad 0<q<1,
  \]
  the bound in~(i) holds with probability at least $1-q$.
\end{enumerate}
\end{proposition}

\begin{proof}
By the proof of Theorem~\ref{thm:response-ident},
$r\mapsto\mathcal{R}_{\Psi}(u,r)$ is strictly decreasing, so
$\mathcal{R}_{\Psi}(u,r)-w_0$ has the sign of $r_*-r$. For~(i), if
$|\tilde v-w_0|>\eta$, then $|e|\le\eta$ forces $\tilde v-w_0$ and
$\mathcal{R}_{\Psi}(u,\mu)-w_0$ to have the same sign, and the update keeps
$r_*$ in the bracket. If the procedure stops at~$\mu$, then
$|\mathcal{R}_{\Psi}(u,\mu)-w_0|\le|\tilde v-w_0|+|e|\le2\eta$. Both $\mu$
and~$r_*$ lie in $[r_L,r_U]$, so the mean-value theorem and the slope bound
give $c_0|\mu-r_*|\le2\eta$. Otherwise, after $N$ halvings the bracket has
length $(r_U-r_L)2^{-N}$ and contains~$r_*$, so its midpoint lies within
$(r_U-r_L)2^{-N-1}$ of~$r_*$. In both cases
$|\log\hat r-\log r_*|\le|\hat r-r_*|/r_L$, because $\hat r,r_*\ge r_L$. For~(ii),
conditional on the history, the $j$-th query error is
$\mathcal{N}(0,\sigma^2/n_{\mathrm{rep}})$, so
$\Pr(|e_j|>\eta)\le2\exp(-n_{\mathrm{rep}}\eta^2/(2\sigma^2))=q/N$. There are
at most $N$ queries, so by the union bound every error satisfies
$|e_j|\le\eta$ with probability at least $1-q$. On that event~(i) applies.
\end{proof}

Proposition~\ref{prop:bisection} provides finite-query recovery guarantees
for selected derivative ratios $g(s)/g(w_0)$ of an unrestricted generator. It
is a target-wise reconstruction result, not an estimator of the entire
function~$g$ over a continuum. Its error bounds require an initial bracket,
exact experimental controls, and a positive lower bound on the response
sensitivity; the location errors are independent of the controls, and the
correlated errors-in-variables models of Section~\ref{sec:numerics} are not
covered. For
$n$ targets, replacing $N$ by $nN$ in~$\eta$ makes the bound hold
simultaneously for every recovered ratio $g(s_i)/g(w_0)$ with probability at
least $1-q$. Supplemental~S12 compares the bound with simulated errors, and
checks Theorem~\ref{thm:chain} below on computed chain equilibria.

The constants $m$, $M$, $g_\pm$, and hence $c_0$ and~$C_0$, are properties of
the unknown generator. Corollaries~\ref{cor:finite-query}
and~\ref{cor:stability} and Proposition~\ref{prop:bisection} are therefore
a~priori statements about the inverse map, not a~posteriori certificates. The
stopping tolerance~$\varepsilon$ of Corollary~\ref{cor:finite-query}, or the
error floor $2\eta/c_0$ of Proposition~\ref{prop:bisection}, yields a
guaranteed log-ratio accuracy for a particular reconstruction only if the
relevant constants are bounded independently, for example by prior
constitutive information. The log-ratio bounds use~$c_0$ only as a lower
bound on $|\partial_r\mathcal{R}_{\Psi}(u,\cdot)|$, a slope of the forward map
that can in principle be measured; we do not analyze an estimator for it. The
Lipschitz constant of~$\Theta$ involves~$g'/g$ on~$K$, which is scale-free but
is itself part of what is being reconstructed.

\subsection{Graph reductions}
\label{sec:graph}

\begin{corollary}[Single-free-vertex reduction on a general finite graph]
\label{cor:graph-reduction}
This reduction requires the following boundary-control protocol: a single free
vertex, every neighbor pinned, one drive edge at~$u$, and all other incident
edges pinned at the reference. The effective ratio is either known and freely
variable, or held fixed at a possibly unknown value different from one.

Let $G=(X,\mathcal{E})$ be a finite weighted graph with a single free vertex~$I$
of weight~$\alpha>0$, and suppose every neighbor of~$I$ is pinned. Pin one
neighbor~$B$ at log-boundary~$u$ with edge weight~$\beta>0$, and pin every
other neighbor at the reference value~$0$, with edge weights~$\beta_j>0$. For
any generator in class~\eqref{eq:Psi-class} the free-coordinate objective equals
\[
\Bigl(\alpha+\sum_j\beta_j\Bigr)\Psi(v)+\beta\,\Psi(u-v)
\;+\;\mathrm{const},
\]
where the additive constant collects the pinned-site terms and any
pinned--pinned edge terms of the ambient graph; being independent of~$v$, it
shifts the objective without moving the minimizer. The problem therefore
reduces to the two-site problem with effective ratio
$r_{\mathrm{eff}}=(\alpha+\sum_j\beta_j)/\beta$. Theorem~\ref{thm:response-ident},
Corollaries~\ref{cor:finite-query}--\ref{cor:stability}, and
Proposition~\ref{prop:bisection} therefore apply under this protocol whenever
$r_{\mathrm{eff}}$ is known and can be varied over the positive range those
statements require. Theorem~\ref{thm:fixed-ratio} applies when
$r_{\mathrm{eff}}\neq1$ is held fixed, possibly at an unknown value, and $u$
is swept.
\end{corollary}

\begin{proof}
Each reference-pinned edge contributes a term $\beta_j\Psi(v-0)=\beta_j\Psi(v)$,
which augments the free-site coefficient. All remaining terms of
$C_{\mathcal{M},\Psi}$ involve pinned coordinates only and are therefore
constant in~$v$. Differentiating in~$v$ annihilates that constant, so
stationarity becomes $r_{\mathrm{eff}}g(v)=g(u-v)$, as claimed.
\end{proof}

Thus the two-site problem is not merely an isolated toy geometry. It is
the local inverse problem seen by any single free vertex whose remaining
neighbors can be pinned at the reference, provided the resulting effective
weight ratio is either known and controllable or fixed and different from one.
The reduction does not reach equilibria with several free vertices. The next
theorem treats one such coupled configuration.

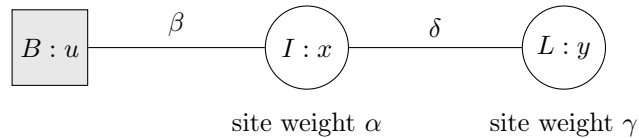
\begin{figure}[htbp]
\centering
\begin{tikzpicture}[
  free/.style={circle,draw=black,fill=white,minimum size=11mm,inner sep=1pt},
  pinned/.style={rectangle,draw=black,fill=gray!18,minimum size=10mm,inner sep=2pt},
  every node/.style={font=\footnotesize}
]
  \node[pinned] (B) at (0,0) {$B:u$};
  \node[free] (I) at (3.4,0) {$I:x$};
  \node[free] (L) at (6.8,0) {$L:y$};
  \draw (B) -- node[above] {$\beta$} (I);
  \draw (I) -- node[above] {$\delta$} (L);
  \node[below=1.6mm of I] {site weight $\alpha$};
  \node[below=1.6mm of L] {site weight $\gamma$};
\end{tikzpicture}
\caption{The two-free-vertex chain of Theorem~\ref{thm:chain}. The boundary~$u$
is swept and both free locations are observed. The leaf~$L$ sees the
location~$x$ of~$I$ as its boundary at the fixed ratio~$\gamma/\delta$.}
\label{fig:chain}
\end{figure}

\begin{theorem}[Identification on a chain with two free vertices]
\label{thm:chain}
Consider the path $B$--$I$--$L$ of Figure~\ref{fig:chain}: $B$ is pinned at
log-boundary $u>0$; the free vertices $I$ and~$L$ have log-coordinates $x$
and~$y$ and site weights $\alpha,\gamma>0$; the edges $BI$ and~$IL$ have
weights $\beta,\delta>0$; and no other edges meet $I$ or~$L$. For $\Psi$ in
class~\eqref{eq:Psi-class} the free-coordinate objective
\[
C_u(x,y)=\alpha\Psi(x)+\beta\Psi(u-x)+\delta\Psi(x-y)+\gamma\Psi(y)
\]
has a unique minimizer $(x(u),y(u))$, and the following hold.
\begin{enumerate}
  \item[(i)] $0<y(u)<x(u)<u$, and $y(u)=\mathcal{R}_{\Psi}(x(u),\gamma/\delta)$.
  \item[(ii)] $u\mapsto x(u)$ is $C^1$ and strictly increasing, with
    $x(u)\to0$ as $u\to0^+$.
  \item[(iii)] $x(u)\to\infty$ as $u\to\infty$ if and only if $g$ is unbounded
    on $(0,\infty)$ or $\beta\ge\alpha+\min\{\gamma,\delta\}$.
\end{enumerate}
When the condition in~(iii) holds and $\gamma\neq\delta$, the observed pairs
$\{(x(u),y(u)):u>0\}$ determine $\gamma/\delta$ and determine~$\Psi$ up to a
positive scale, uniquely after the calibration $g'(0)=1$. The boundary must be
swept over $(0,\infty)$, but its values need not be recorded, and the weights
need not be known; they must only stay fixed during the sweep.
\end{theorem}

\begin{proof}
Existence and uniqueness follow from Proposition~\ref{prop:P3}, and the
stationarity equations are
\[
\alpha g(x)-\beta g(u-x)+\delta g(x-y)=0,\qquad
\gamma g(y)-\delta g(x-y)=0.
\]
(i) Write $r_{\mathrm{leaf}}:=\gamma/\delta$. The second equation,
$r_{\mathrm{leaf}}\,g(y)=g(x-y)$, is the two-site stationarity condition with
boundary~$x$ and ratio~$r_{\mathrm{leaf}}$. Since $g$ is odd and increasing, $y$
and~$x-y$ have the same sign. If $y\le0$, then $x\le y\le0$, so
$\alpha g(x)+\delta g(x-y)\le0<\beta g(u-x)$, contradicting the first
equation. Hence $y>0$ and $x-y>0$, and the first equation gives
$\beta g(u-x)>0$, so $x<u$. Uniqueness of the two-site minimizer gives
$y=\mathcal{R}_{\Psi}(x,r_{\mathrm{leaf}})$.

(ii) The Hessian of~$C_u$,
\[
H=\begin{pmatrix}
\alpha g'(x)+\beta g'(u-x)+\delta g'(x-y) & -\delta g'(x-y)\\
-\delta g'(x-y) & \gamma g'(y)+\delta g'(x-y)
\end{pmatrix},
\]
is positive definite because $g'>0$. The implicit-function theorem therefore
makes $(x,y)$ a $C^1$ function of~$u$, with
$H\,(x',y')^{\top}=(\beta g'(u-x),0)^{\top}$, so
$x'=\beta g'(u-x)\,H_{22}/\det H>0$. The bound $0<x(u)<u$ gives $x\to0$ as
$u\to0^+$.

(iii) Let $\phi(x):=\alpha g(x)+\delta g\bigl(x-\mathcal{R}_{\Psi}(x,r_{\mathrm{leaf}})\bigr)$,
so that by~(i) the first equation reads $\beta g(u-x(u))=\phi(x(u))$.
Differentiating $r_{\mathrm{leaf}}\,g(R)=g(x-R)$ with
$R:=\mathcal{R}_{\Psi}(x,r_{\mathrm{leaf}})$ gives
$\partial_xR=g'(x-R)/\bigl(r_{\mathrm{leaf}}\,g'(R)+g'(x-R)\bigr)\in(0,1)$, so
both $R$ and $x-R$ are $C^1$ and strictly increasing in~$x$, and $\phi$ is
continuous and strictly increasing. Let
$g_\infty:=\sup_{t>0}g(t)\in(0,\infty]$ and $x_\infty:=\lim_{u\to\infty}x(u)\in(0,\infty]$,
which exists by~(ii). If $x_\infty<\infty$, then $u-x(u)\to\infty$, and letting
$u\to\infty$ gives $\beta g_\infty=\phi(x_\infty)<\infty$, so $g_\infty<\infty$. As $x\to\infty$,
the larger of $\mathcal{R}_{\Psi}(x,r_{\mathrm{leaf}})$ and $x-\mathcal{R}_{\Psi}(x,r_{\mathrm{leaf}})$
is at least $x/2$. Hence $g(x-\mathcal{R}_{\Psi}(x,r_{\mathrm{leaf}}))\to g_\infty$ if
$r_{\mathrm{leaf}}\ge1$, and
$g(x-\mathcal{R}_{\Psi}(x,r_{\mathrm{leaf}}))=r_{\mathrm{leaf}}\,g(\mathcal{R}_{\Psi}(x,r_{\mathrm{leaf}}))\to r_{\mathrm{leaf}} g_\infty$
if $r_{\mathrm{leaf}}<1$. Thus $\sup\phi=(\alpha+\min\{\gamma,\delta\})g_\infty$, and the supremum is
not attained. So $x_\infty<\infty$ implies $\beta g_\infty=\phi(x_\infty)<(\alpha+\min\{\gamma,\delta\})g_\infty$,
that is, $g_\infty<\infty$ and $\beta<\alpha+\min\{\gamma,\delta\}$. Conversely, if
$g_\infty<\infty$, $\beta<\alpha+\min\{\gamma,\delta\}$, and $x_\infty=\infty$, then
$\phi(x(u))\to(\alpha+\min\{\gamma,\delta\})g_\infty>\beta g_\infty\ge\beta g(u-x(u))$,
contradicting the first equation for large~$u$.

Under the condition in~(iii), $x$ is a continuous increasing bijection of
$(0,\infty)$ onto itself. By~(i) the observed pairs are therefore exactly the
graph of the fixed-ratio response $x\mapsto\mathcal{R}_{\Psi}(x,r_{\mathrm{leaf}})$ on
$(0,\infty)$, and Theorem~\ref{thm:fixed-ratio} with $\bar u=\infty$ applies
when $r_{\mathrm{leaf}}\neq1$.
\end{proof}

The reconstruction uses only the leaf equation; the coupled first equation
serves to guarantee the sweep. The sweep criterion in
Theorem~\ref{thm:chain}(iii) assumes fixed weights and guarantees that the
observed interior coordinate~$x$ ranges over the positive half-line; the
criterion involves unobserved weights, but the range of~$x$ actually swept is
observed. Once that range condition is established independently,
reconstruction from the leaf equation requires only that the ratio
$\gamma/\delta$ remain fixed and unequal to one. Thus $\alpha$ and~$\beta$ may
vary across observations without invalidating the reconstruction argument,
provided the required range of~$x$ is attained. Such variation is not covered
by the sweep criterion proved in part~(iii). When $g$ is bounded and
$\beta<\alpha+\min\{\gamma,\delta\}$, $x(u)$ increases to the finite limit
$x_\infty=\phi^{-1}(\beta g_\infty)$ without reaching it. For
$\gamma\neq\delta$, applying the local form of Theorem~\ref{thm:fixed-ratio}
on $(0,\bar u]$ for each $\bar u<x_\infty$ then identifies~$g$ up to scale on
$(0,s_\infty)$, where $s_\infty$ is the larger of
$\mathcal{R}_{\Psi}(x_\infty,\gamma/\delta)$ and
$x_\infty-\mathcal{R}_{\Psi}(x_\infty,\gamma/\delta)$. The same argument
applies to any free leaf joined by a single edge to a vertex whose observed
coordinate takes every value in an interval $(0,\bar x)$, provided the leaf's
site-to-edge weight ratio is fixed and different from one. Compared with
Corollary~\ref{cor:graph-reduction}, the leaf's neighbor is free and observed
rather than pinned and controlled. Errors in the observed~$x$ take the place
of boundary errors in the leaf equation, and their effect on the fixed-ratio
reconstruction is not analyzed. Identification on graphs without such a leaf,
or where the sweep cannot be established, remains open.

\subsection{Equal-weight non-identifiability}
\label{sec:symmetric-design}
On the equal-weight two-site chain ($r=1$), every strictly convex even
log-coordinate cost yields an objective of the form $\Psi(v)+\Psi(u-v)$ up to a
constant. Differentiating gives $\Psi'(v)-\Psi'(u-v)=0$, hence
$v^\star=u/2$ independently of~$\Psi$. Location alone cannot identify the
generator under reflection symmetry; asymmetric weight ratios are essential.
Details, including the two-site closed form under unequal weights, appear in
Supplemental~S5.

\section{The reciprocal candidate $J$ and the family $J_p$}
\label{sec:reciprocal}

The inverse theorem and the RCL characterization answer different questions.
Section~\ref{sec:ident} asks what generator the response map determines
without assuming a particular functional form. RCL instead supplies an
independently motivated candidate,~$J$, whose predicted response can then
be compared with that reconstruction or with finite-data alternatives.
Proposition~\ref{prop:J-unique} records that characterization via the
classical d'Alembert classification~\cite{Aczel1966}; it is included so
that the subsequent comparison with~$J$ and the family~$J_p$ rests on an
independent selection principle rather than on an a~priori restriction of
the inverse theorem.

\begin{proposition}[Background: uniqueness of the reciprocal cost under RCL]
\label{prop:J-unique}
\label{thm:J-unique}
Let $F:\mathbb{R}_{>0}\to\mathbb{R}$ be continuous and satisfy reciprocity
$F(x)=F(x^{-1})$, normalization $F(1)=0$, the ratio-composition law
\[
H(xy)+H(x/y)=2H(x)H(y),\qquad H:=F+1,
\]
and the unit log-curvature calibration as a punctured limit
\[
\lim_{\substack{t\to0\\ t\neq0}}\frac{2F(e^t)}{t^2}=1
\]
(the quotient is undefined at $t=0$). Then $F=J$, where
\begin{equation}
\label{eq:J-def}
J(x)=\tfrac12\bigl(x+x^{-1}\bigr)-1=\cosh(\log x)-1.
\end{equation}
\end{proposition}

\begin{proof}[Proof sketch]
With $\mathcal{H}(t):=F(e^t)+1$, reciprocity and normalization make~$\mathcal{H}$
continuous and even with $\mathcal{H}(0)=1$, and RCL becomes d'Alembert's
equation $\mathcal{H}(t+u)+\mathcal{H}(t-u)=2\mathcal{H}(t)\mathcal{H}(u)$. Its
continuous even solutions with $\mathcal{H}(0)=1$ are $\cosh(kt)$, $\cos(kt)$,
and the constant~$1$~\cite{Aczel1966}, and the calibration
$\mathcal{H}''(0)=1$ selects $\cosh t$. The full argument is Supplemental
Theorem~S1; see also~\cite{WashburnAczel}.
\end{proof}

\begin{corollary}[Propagation of scalar uniqueness under a pairwise homogeneous ansatz]
\label{thm:graph-propagation}
Let $C:(\mathbb{R}_{>0})^{X}\to\mathbb{R}_{\ge0}$ be a deviation functional on a
finite weighted graph satisfying:
\begin{enumerate}
  \item[\textup{(L)}] \emph{Locality.} There are continuous functions
    $f_s\ge0$ and $h_{ss'}\ge0$ with
    $C[m]=\sum_{s}f_s(m_s)+\sum_{\{s,s'\}}h_{ss'}(m_s,m_{s'})$.
  \item[\textup{(G)}] \emph{Ratio-invariant, reciprocal edges.} Each edge term
    depends on its endpoints only through their ratio and is inversion-invariant:
    $h_{ss'}(x,y)=\theta_{ss'}(x/y)$ with $\theta_{ss'}(r)=\theta_{ss'}(1/r)$.
  \item[\textup{(H)}] \emph{Homogeneity.} A single continuous scalar shape
    $\phi:\mathbb{R}_{>0}\to\mathbb{R}_{\ge0}$ governs every local term:
    $f_s=\omega_s\phi$ and $\theta_{ss'}=\eta_{ss'}\phi$ for constants
    $\omega_s,\eta_{ss'}>0$.
  \item[\textup{(A)}] \emph{Scalar axioms.} $\phi$ satisfies the hypotheses of
    Proposition~\ref{prop:J-unique}.
\end{enumerate}
Then $\phi=J$, and writing $w_s=\omega_s$, $w_{ss'}=\eta_{ss'}$, one has
$C=C_{\mathcal{M},\Psi}$ with $\Psi(t)=\cosh t-1$. Normalization is a
consequence rather than a hypothesis: $C[m]=0$ if and only if
$m\equiv1$. Within the class fixed by (L), (G), and~(H), the scalar axioms
leave only the weights as free parameters.
\end{corollary}

\begin{proof}
By (L), (G), and~(H),
\[
C[m]
  =\sum_{s}\omega_s\,\phi(m_s)
  +\sum_{\{s,s'\}}\eta_{ss'}\,\phi\!\left(\frac{m_s}{m_{s'}}\right).
\]
Hypothesis~(A) and Proposition~\ref{prop:J-unique} give $\phi=J$. Renaming
weights yields $C=C_{\mathcal{M},\Psi}$ with $\Psi(t)=\cosh t-1$. For
normalization, $J\ge0$ with $J(x)=0$ only at $x=1$, so every summand of the
displayed expression is nonnegative; since all $\omega_s>0$, the total vanishes
only when $m_s=1$ at every site, and the uniform field indeed gives $C=0$.
\end{proof}

Equivalently $J(x)=(x-1)^2/(2x)$. The cost is nonnegative, strictly convex, and
diverges as $x\to0^+$ and as $x\to\infty$; $2J(x/y)$ is the additive
symmetrization of the Itakura--Saito (COSH)
divergence~\cite{CichockiAmari2010,BanerjeeMBG2005,Amari2016}. Multiplicative
data alone do not imply RCL; the law is an additional all-scale assumption that
distinguishes~$J$ from nearby unit-curvature rivals such as
$Q(x)=\tfrac12(\log x)^2$ and the one-parameter family
\begin{equation}
\label{eq:Jp-def}
J_p(x)=\frac{\cosh(p\log x)-1}{p^2}\quad(p>0),
\end{equation}
with $J_1=J$ and $J_p\to Q$ pointwise as $p\to0$. Writing
$\Psi_p(t)=(\cosh(pt)-1)/p^2$ one has $g_p=\Psi_p'$ and
$g_p'(0)=\Psi_p''(0)=\cosh(0)=1$ for every $p>0$, so each~$J_p$ lies in the
unit-curvature calibrated class of Theorem~\ref{thm:response-ident}.
Corollary~\ref{thm:graph-propagation} records the termwise
propagation of Proposition~\ref{prop:J-unique} under a pairwise homogeneous
graph ansatz.

For~$J_p$ the two-site stationarity condition is elementary. Writing
$r=\alpha/\beta$, one finds $\alpha\sinh(pv)=\beta\sinh(p(u-v))$, and the
substitution $X=e^{pv}$ yields the closed-form response
\begin{equation}
\label{eq:vstar-Jp}
v^\star_p(u,r)=\frac{1}{2p}\,
  \log\!\left(\frac{r+e^{pu}}{r+e^{-pu}}\right).
\end{equation}
At $r=1$ one has $v^\star_p=u/2$ for every~$p$, recovering the equal-weight
non-identifiability of Section~\ref{sec:symmetric-design}.

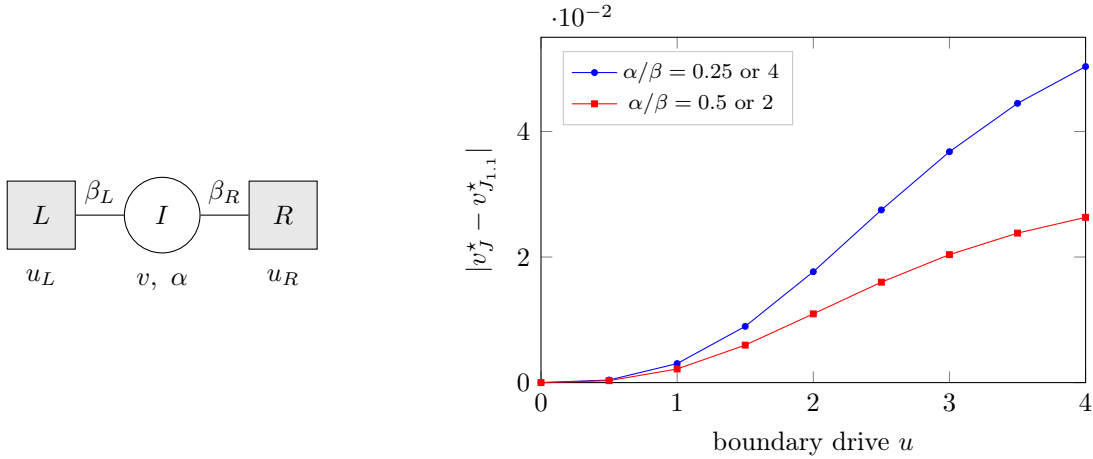
\begin{figure}[htbp]
\centering
\begin{minipage}[c]{0.40\linewidth}
\centering
\begin{tikzpicture}[
  free/.style={circle,draw=black,fill=white,minimum size=10mm,inner sep=1pt},
  pinned/.style={rectangle,draw=black,fill=gray!18,minimum size=9mm,inner sep=2pt},
  every node/.style={font=\footnotesize}
]
  \node[pinned] (L) at (0,0) {$L$};
  \node[free] (C) at (1.6,0) {$I$};
  \node[pinned] (R) at (3.2,0) {$R$};
  \draw (L) -- node[above] {$\beta_L$} (C);
  \draw (C) -- node[above] {$\beta_R$} (R);
  \node[below=1.2mm of L] {$u_L$};
  \node[below=1.2mm of C] {$v,\ \alpha$};
  \node[below=1.2mm of R] {$u_R$};
\end{tikzpicture}
\end{minipage}\hfill
\begin{minipage}[c]{0.56\linewidth}
\centering
\begin{tikzpicture}
\begin{axis}[
  width=\linewidth,height=0.70\linewidth,
  xlabel={boundary drive $u$},
  ylabel={$|v^\star_J-v^\star_{J_{1.1}}|$},
  xmin=0,xmax=4,ymin=0,ymax=0.055,
  legend style={at={(0.04,0.96)},anchor=north west,font=\scriptsize,
    draw=gray!45,fill=white},
  tick label style={font=\footnotesize},label style={font=\footnotesize}
]
\addplot[blue,mark=*,mark size=1.1pt] coordinates {
  (0,0)(0.5,0.00041)(1,0.00303)(1.5,0.00896)(2,0.01766)
  (2.5,0.02750)(3,0.03677)(3.5,0.04448)(4,0.05034)};
\addlegendentry{$\alpha/\beta=0.25$ or $4$}
\addplot[red,mark=square*,mark size=1.1pt] coordinates {
  (0,0)(0.5,0.00031)(1,0.00217)(1.5,0.00597)(2,0.01095)
  (2.5,0.01600)(3,0.02038)(3.5,0.02381)(4,0.02632)};
\addlegendentry{$\alpha/\beta=0.5$ or $2$}
\end{axis}
\end{tikzpicture}
\end{minipage}
\caption{Left: three-site discriminator. Right: two-site equilibrium-location
separation between~$J$ and~$J_{1.1}$ under asymmetric weight ratios.}
\label{fig:Jp-separation}
\end{figure}

\subsection{Secondary three-site discriminator}
\label{sec:threesite-local}
As a local contrast at one operating point, an asymmetric three-site chain
separates~$J$ from the quadratic-log comparator~$Q$ at cubic order near the
reference.

\begin{proposition}[Local three-site separation of $J$ from $Q$]
\label{prop:threesite}
The one-free-site objective for a scalar cost with $\Psi(t):=F(e^t)$ is
\begin{equation}
\label{eq:threesite-obj}
C_I(v)=\alpha\,\Psi(v)+\beta_L\,\Psi(u_L-v)+\beta_R\,\Psi(v-u_R).
\end{equation}
For $J$ ($\Psi=\cosh-1$) the minimizer $v_J$ solves
\begin{equation}
\label{eq:threesite-J}
\alpha\sinh v-\beta_L\sinh(u_L-v)+\beta_R\sinh(v-u_R)=0,
\end{equation}
whereas for the quadratic-log comparator $Q$ ($\Psi=\tfrac12 t^2$) it solves
\begin{equation}
\label{eq:threesite-Q}
\alpha v-\beta_L(u_L-v)+\beta_R(v-u_R)=0,
\qquad
v_Q=\frac{\beta_L u_L+\beta_R u_R}{\alpha+\beta_L+\beta_R}.
\end{equation}
The two objectives share the same quadratic Taylor polynomial at the reference,
while their minimizer maps first differ at cubic order in the boundary
amplitudes; $v_J\neq v_Q$ for all parameters outside a measure-zero set.
\end{proposition}

\begin{proof}
Both $\cosh-1$ and $\tfrac12 t^2$ are strictly convex and the arguments
in~\eqref{eq:threesite-obj} are affine in~$v$, so each $C_I$ is strictly convex
with a unique critical point, giving
\eqref{eq:threesite-J}--\eqref{eq:threesite-Q}. By strict convexity of the
$J$-objective, $v_J=v_Q$ holds if and only if $v_Q$ is also its critical
point, i.e.\ if and only if
\[
\mathcal{A}(u_L,u_R,\alpha,\beta_L,\beta_R)
  :=\alpha\sinh v_Q-\beta_L\sinh(u_L-v_Q)+\beta_R\sinh(v_Q-u_R)=0,
\]
with $v_Q$ given by~\eqref{eq:threesite-Q}. The vanishing of the $O(\|u\|^3)$
term alone does not establish this; we argue instead from analyticity. The
map~$\mathcal{A}$ is real-analytic in $(u_L,u_R,\alpha,\beta_L,\beta_R)$ on
$\alpha,\beta_L,\beta_R>0$ and is not identically zero: at
$\alpha=\beta_L=\beta_R=1$, $u_R=0$, $u_L=3s$ one has $v_Q=s$ and
\[
\mathcal{A}=2\sinh s-\sinh 2s=2\sinh s\,(1-\cosh s),
\]
which is nonzero for every $s\neq0$. The zero set of a real-analytic function that is not identically zero has empty
interior and Lebesgue measure zero~\cite{KrantzParks2002}, so
$\mathcal{A}\neq0$, and hence $v_J\neq v_Q$, for all
$(u_L,u_R,\alpha,\beta_L,\beta_R)$ outside a measure-zero set. To locate the leading separation, substitute
$\sinh t=t+t^3/6+\cdots$ into $\mathcal{A}$. The linear part vanishes by the
definition of~$v_Q$ in~\eqref{eq:threesite-Q}, leaving the cubic residual
\[
\mathcal{A}(v_Q)
=\frac{\alpha v_Q^3-\beta_L(u_L-v_Q)^3+\beta_R(v_Q-u_R)^3}{6}
+O(\|u\|^5),
\]
where $\|u\|:=\max\{|u_L|,|u_R|\}$. Writing
$\mathcal{A}_J(v):=\alpha\sinh v-\beta_L\sinh(u_L-v)+\beta_R\sinh(v-u_R)$
for the $J$-stationarity residual as a function of~$v$, one has
$\mathcal{A}_J(v_Q)=\mathcal{A}(u_L,u_R,\alpha,\beta_L,\beta_R)$ and
$\mathcal{A}_J'(v_Q)=\alpha+\beta_L+\beta_R+O(\|u\|^2)$. Taylor expansion of
the stationarity equation about~$v_Q$, together with that derivative estimate,
yields
\[
v_J-v_Q
=\frac{-\alpha v_Q^3+\beta_L(u_L-v_Q)^3-\beta_R(v_Q-u_R)^3}
       {6(\alpha+\beta_L+\beta_R)}+O(\|u\|^5).
\]
In the special case $\alpha=\beta_L=\beta_R=1$, $u_R=0$, $u_L=3s$ one has
$v_Q=s$ and leading displacement $v_J-v_Q=s^3/3+O(s^5)$. Thus the minimizer
maps agree through quadratic order in~$\|u\|$ and first differ at cubic order. The objectives share their quadratic Taylor polynomial and first
differ at quartic order, since $J(e^t)=\cosh t-1=t^2/2+t^4/24+O(t^6)$ and
$Q(e^t)=t^2/2$ both have vanishing cubic term; the induced minimizer maps
nonetheless first differ at cubic order in the boundary amplitudes, because
stationarity differentiates the objective.
\end{proof}

The cubic near-reference separation accounts for the local proximity of $J$
and~$Q$. The numerical design of Section~\ref{sec:numerics} works away from that
shared near-reference regime. Stationarity conditions for additional comparators
are recorded in Supplemental~S6.

\section{Finite-sweep discrimination of generators}
\label{sec:numerics}

The experiments below are synthetic and belong to level~(b) of
Section~\ref{sec:levels}: finite-data discrimination within the calibrated
family~$J_p$ under a specified observation model, not nonparametric
reconstruction of~$g$. Table~\ref{tab:design} collects the full specification of the primary design.
A secondary three-site discriminator is summarized at the end, with full
details in Supplemental~S10.

\begin{table}[htbp]
\centering
\caption{Specification of the synthetic primary sweep (Section~\ref{sec:primary-sweep}).
All quantities were fixed before any synthetic draws.}
\label{tab:design}
\footnotesize
\begin{tabular}{@{}>{\raggedright\arraybackslash}p{0.22\linewidth}p{0.72\linewidth}@{}}
\toprule
Item & Specification \\
\midrule
Geometry
& One free site joined by a single edge to a neighbor pinned at~$u$; weight
ratio $r=\alpha/\beta$; equilibria from~\eqref{eq:gadget-stat}. \\
\addlinespace
Design grid
& $(u,r)\in\{2,3,4\}\times\{0.25,0.5,1,2,4\}$ ($15$ conditions). The three
$r=1$ conditions carry no information about~$p$ and are symmetry controls. \\
\addlinespace
Observation model
& $n=7$ independent replicates per condition; additive Gaussian noise on the
log-location~$v$ with $\sigma=0.05$; $105$ observations per dataset. \\
\addlinespace
Candidates
& $J$ ($p=1$) against a fully specified~$J_p$,
$p\in\{1.02,1.05,1.10,1.15,1.20\}$ (simple versus simple). \\
\addlinespace
Decision rule
& Select~$J$ if and only if $\mathrm{SSE}_J<\mathrm{SSE}_{J_p}$ over all $105$
observations, scored against nominal predictions; ties are not counted
for~$J$. \\
\addlinespace
Errors-in-variables screen
& Shared and condition-specific Gaussian errors on~$u$ and on~$\log r$ with
total standard deviation~$w$ and variance split $0.7{:}0.3$; the same realized
errors enter both data-generating models. \\
\addlinespace
Estimator of~$p$
& Least squares on condition means over the grid $p\in[0.85,1.35]$, spacing
$0.005$; Wald standard error fixed at $I(1)^{-1/2}$. Its nuisance simulation
splits the variance $0.5{:}0.5$. \\
\addlinespace
Forward solver
& Bounded scalar minimization (selection) or Brent root-finding (screens,
estimator); agreement with the closed form~\eqref{eq:vstar-Jp} within
$3\times10^{-8}$ on the grid. $\partial_p v^\star$ by central differences,
$h=10^{-4}$. \\
\addlinespace
Monte Carlo
& $N=2\times10^4$ datasets (selection, headline screens); $4000$ per width
(width scans); $2000$ and $1500$ (estimator, exact and $w=0.10$). Seeds
\texttt{20260713} (selection, screens) and \texttt{20260730} (estimator);
NumPy \texttt{default\_rng} (PCG64). \\
\bottomrule
\end{tabular}
\end{table}

\subsection{Primary two-site boundary/weight sweep}
\label{sec:primary-sweep}
The design is the minimal asymmetric configuration: one free site joined by a
single edge to a neighbor pinned at boundary multiplier $a=e^{u}$, with no other
edge incident on the free site. Equilibria are the unique minimizers of the
two-site objective, characterized by~\eqref{eq:gadget-stat}. Because each candidate
generator predicts a full vector of equilibrium locations across a
prespecified grid, the natural analyses are (i)~pairwise model selection between
fully specified generators by aggregate residual, and (ii)~least-squares
estimation of the reciprocal-family index~$p$ (with $p=1$ giving~$J$).

By~\eqref{eq:vstar-Jp}, predicted locations, their derivatives in~$p$, and
hence the Fisher information below are available in closed form. The archived
code evaluates them numerically (Table~\ref{tab:design}); evaluating
$\partial_p v^\star_p$ analytically from~\eqref{eq:vstar-Jp} reproduces the
reported Fisher information to four significant figures. At $r=1$ one has
$v^\star_p=u/2$ for every~$p$, so equal-weight conditions are retained only as
symmetry controls.

\paragraph{Decision rule and observation-noise-only pairwise selection}
Fix in advance the grid $u\in\{2,3,4\}$ crossed with
$\alpha/\beta\in\{0.25,0.5,1,2,4\}$ ($15$ points), with $n=7$ independent
replicates per condition, each the true log-location plus Gaussian noise of
standard deviation~$\sigma=0.05$. Selection simulations used $N=2\times10^4$
synthetic datasets and seed~\texttt{20260713} with NumPy's
\texttt{default\_rng} (bit generator \texttt{PCG64}); Monte Carlo sizes and
seeds for every reported quantity are listed in Table~\ref{tab:design}.
The residual statistic is the total sum of squared errors (SSE) over all
$15\times7$ replicate observations against each candidate's nominal closed-form
predictions~\eqref{eq:vstar-Jp}; residuals are not computed from
condition means alone. Candidate~$J$ is selected when
$\mathrm{SSE}_J<\mathrm{SSE}_{J_p}$; equal SSE is a tie and is not counted as a
selection of~$J$ (margin~$0$). By~\eqref{eq:vstar-Jp} the conditions with
$\alpha/\beta=1$ predict the same location for every member of the family, so
they carry zero local Fisher information about~$p$
($\partial_p v^\star_p(u,1)=0$ for every $u$ and $p$) and contribute identical
residuals to both candidates, so they cancel exactly from any pairwise SSE
difference. The effective informative design is therefore the asymmetric
$12$-point sub-grid $\alpha/\beta\in\{0.25,0.5,2,4\}$; the three equal-weight
conditions are retained only as symmetry controls and add no discriminating
power to either the selection statistic or the Fisher information for~$p$. Under this rule the
observation-noise-only simulated selection probabilities are those of
Table~\ref{tab:p-selection}: $0.997$ at $p=1.1$, $0.921$ at $p=1.05$, and
$0.720$ at $p=1.02$. These figures assume every boundary~$u$ and weight
ratio~$\alpha/\beta$ is known exactly, which is what ``observation-noise-only''
means here and for the three-site $0.828$ figure.

In this observation-noise-only setting the selection probability is also
available exactly, which provides an analytical benchmark for the simulation.
Write $\delta_k:=v^\star_{J_p}(u_k,r_k)-v^\star_J(u_k,r_k)$. When~$J$
generates the data, $\mathrm{SSE}_J-\mathrm{SSE}_{J_p}
=2\sum_k\delta_kS_k-n\sum_k\delta_k^2$, where $S_k$ is the sum of the~$n$
noise draws at condition~$k$, so the difference is Gaussian with mean
$-n\sum_k\delta_k^2$ and variance $4n\sigma^2\sum_k\delta_k^2$. Hence
\begin{equation}
\label{eq:exact-selection}
\Pr(\text{select }J\mid J\text{ true})
=\Phi\!\left(\frac{\sqrt{n\sum_k\delta_k^2}}{2\sigma}\right),
\end{equation}
with $\Phi$ the standard normal distribution function; by symmetry
$\Pr(\text{select }J\mid J_p\text{ true})$ is one minus this value. Evaluated
with the closed form~\eqref{eq:vstar-Jp}, \eqref{eq:exact-selection} gives
$0.7161$, $0.9206$, $0.9971$, $0.99997$, and $1-7\times10^{-8}$ for
$p=1.02,\ldots,1.20$ (Table~\ref{tab:p-selection}), and every simulated entry
lies within $1.2$ Monte Carlo standard errors of the exact value. The
observation-noise-only pairwise selection probabilities can therefore be
evaluated analytically; Monte Carlo serves as an independent implementation
check. Simulation remains useful for assessing estimator behavior and for the
errors-in-variables experiments, whose nuisance distributions introduce
additional uncertainty.

\begin{table}[!htbp]
\centering
\caption{Observation-noise-only selection probabilities, simulated and exact, for pairwise
comparison of~$J$ with a fully specified rival~$J_p$ in the $15$-point
single-exchange sweep ($n=7$ independent replicates per
condition, $\sigma=0.05$, $N=2\times10^4$,
seed~\texttt{20260713}). The first row reports the simulated probability of
selecting~$J$ when~$J$ generates the data; the second reports selection of~$J$
when~$J_p$ generates the data. This is simple-versus-simple model selection, not a composite test of
$\{p\neq1\}$. The Monte Carlo row gives the binomial standard error, common to
both simulated rows; the exact row is~\eqref{eq:exact-selection}.}
\label{tab:p-selection}
\begin{tabular}{lccccc}
\toprule
$p$ & $1.02$ & $1.05$ & $1.10$ & $1.15$ & $1.20$ \\
\midrule
$\Pr(\text{select }J\mid J\text{ true})$ & $0.720$ & $0.921$ & $0.997$ & $1.000$ & $1.000$ \\
$\Pr(\text{select }J\mid J_p\text{ true})$ & $0.282$ & $0.078$ & $0.003$ & $0.000$ & $0.000$ \\
Monte Carlo s.e. & $0.003$ & $0.002$ & $0.0004$ & $<0.001$ & $<0.001$ \\
Exact, \eqref{eq:exact-selection} & $0.7161$ & $0.9206$ & $0.9971$ & $0.99997$ & $>0.99999$ \\
\midrule
$\max_k|v^\star_J-v^\star_{J_p}|$ & $0.011$ & $0.026$ & $0.050$ & $0.073$ & $0.095$ \\
\bottomrule
\end{tabular}
\end{table}
\FloatBarrier

\paragraph{Representation-uncertainty screening for the primary sweep}
The three-site representation-uncertainty analysis does not transfer to
this geometry. We therefore re-simulated the sweep under a screening
errors-in-variables model: for each Monte~Carlo draw and each grid point~$k$,
\[
u_k^{\mathrm{true}}=u_k+\varepsilon^{\mathrm{sh}}_u+\varepsilon^{(k)}_u,\qquad
(\alpha/\beta)_k^{\mathrm{true}}
=(\alpha/\beta)_k\exp\bigl(\varepsilon^{\mathrm{sh}}_{\rho}+\varepsilon^{(k)}_{\rho}\bigr),
\]
with independent zero-mean Gaussians whose shared and condition-specific
variances split~$w^2$ in the proportions $0.7{:}0.3$, that is
$\varepsilon^{\mathrm{sh}}\sim\mathcal{N}(0,0.7\,w^2)$ and
$\varepsilon^{(k)}\sim\mathcal{N}(0,0.3\,w^2)$. The two levels are nested
explicitly: within one simulated dataset a single shared draw
$\varepsilon^{\mathrm{sh}}_u$ (and one $\varepsilon^{\mathrm{sh}}_{\rho}$)
applies to all $15$ grid points, whereas the condition-specific draws
$\varepsilon^{(k)}$ are independent across grid points but common to the $7$
replicates at a point; only the additive observation noise~$\sigma$ is
independent replicate to replicate. The same realized nuisance enters both
candidate generators' data-generating maps, while the analyst scores SSE
against the nominal closed-form predictions (Supplemental~S9.3). The SSE
decision rule and tie convention above are unchanged. At the nominal
width $w=0.10$, the simulated probability of selecting~$J$ against~$J_{1.1}$
falls from~$0.997$ to~$\approx0.88$;
under the same screen, the model-specific probability of selecting~$J$ when
$J_{1.1}$ generates the data is~$\approx0.10$. At $w=0.15$ it
is~$\approx0.79$. A reduced
$12$-point grid that drops the $\alpha/\beta=1$ controls gives $0.877$ at
$w=0.10$; because those conditions cancel exactly from the pairwise statistic,
the gap from~$0.882$ is Monte Carlo variation at $N=4000$ per width, not a loss
of information. Shared-only errors are milder than
condition-specific ones of the same width ($0.97$ versus $0.77$ at $w=0.10$);
adding a batch-level random effect of size~$0.03$ on the replicates does not
materially change the $w=0.10$ figure. Against nearer rivals the degradation is
stronger: at $w=0.10$, the screening selection probability is~$\approx0.72$ at $p=1.05$ and
$\approx0.59$ at $p=1.02$. Figure~\ref{fig:primary-sweep} summarizes the primary
sweep. Headline probabilities used $N=2\times10^4$ Monte Carlo datasets
(binomial Monte Carlo standard error at most~$0.004$); width scans used
$N=4000$ per point (standard error at most~$0.008$).

\paragraph{Least-squares estimation of $p$}
Table~\ref{tab:p-selection} reports simple-versus-simple selection probabilities
against fixed~$J_p$ alternatives. Separately, under the exact-parameter Gaussian
observation model the local Fisher information is constructed in closed form from
$\partial_p v^\star_p$ via~\eqref{eq:vstar-Jp}: for independent replicates of
variance~$\sigma^2$,
\[
I(p)=\sigma^{-2}\sum_{k=1}^{15}n\bigl(\partial_p v^\star_p(u_k,r_k)\bigr)^2,
\]
so at $p=1$ the Cram\'er--Rao lower bound is
$\operatorname{sd}(\hat p)\ge1/\sqrt{I(1)}=0.017$ for the $15$-point sweep
(versus~$0.13$ at the single three-site point). Direct Monte Carlo of the
least-squares estimator minimized the sum of squared deviations between the
$15$ condition means and predicted locations on the fixed grid
$p\in[0.85,1.35]$ with spacing~$0.005$; ties are broken by taking the smallest grid~$p$ attaining the minimum
(\texttt{argmin}), and boundary solutions are kept at the corresponding
endpoint. The Wald standard error used for the reported coverage and for the
test of $H_0{:}\,p=1$ is the fixed value $1/\sqrt{I(1)}$, evaluated at $p=1$
rather than at each fitted~$\hat p$; this is the idealized exact-parameter
construction, and it is precisely why the intervals undercover once nuisance
errors are admitted. With parameters known exactly the estimator attained that
scale over $N=2000$ datasets (seed~\texttt{20260730}): bias $-0.0004$,
empirical sd~$0.0172$, Wald $95\%$ coverage~$0.943$, and type-I error of
$H_0{:}\,p=1$ equal to~$0.057$ (binomial Monte Carlo standard error
about~$0.005$). The simulated rejection proportion for $H_0$ at
$p_{\mathrm{true}}=1.1$ was $1.000$. Across three independent seeds, the
exact-parameter intervals attained coverage $0.934$--$0.938$ (approximately
$94\%$), modest undercoverage partly attributable to the estimator's $0.005$
grid; this finite-grid effect is distinct from the much larger coverage loss
under calibration uncertainty. The estimator's nuisance simulation uses
errors of the same form as the selection screen above, with total width
$w=0.10$ but with shared and condition-specific variances split equally
rather than $0.7{:}0.3$ ($N=1500$). Under that model the empirical standard
deviation inflated to~$0.050$, and the naive exact-parameter Wald intervals
covered $p_{\mathrm{true}}=1$ in only~$0.50$ of datasets, so the nominal
$5\%$ test of $H_0{:}\,p=1$ rejected a true null in~$0.50$ of datasets.
Exact-parameter inference is therefore invalid once calibration errors of
this size are present. The~$0.017$ figure is a lower bound under exact
parameters, and valid interval estimation under nuisance uncertainty would
require profiling, integrating, or bootstrapping those errors.

\begin{figure}[htbp]
\centering
\begin{minipage}[t]{0.485\linewidth}
\centering
\begin{tikzpicture}
\begin{axis}[
  width=0.86\linewidth, height=0.70\linewidth,
  xlabel={family index $p$},
  ylabel={$\Pr(\mathrm{select}\ J\mid J)$},
  xmin=1.0, xmax=1.22, ymin=0.45, ymax=1.05,
  legend style={at={(0.5,-0.40)}, anchor=north, font=\footnotesize,
    draw=gray!45, fill=white, legend columns=-1},
  tick label style={font=\footnotesize}, label style={font=\footnotesize},
  error bars/error bar style={line width=0.5pt},
]
\addplot[blue,mark=*,mark size=1.6pt,
  error bars/.cd, y dir=both, y explicit] coordinates {
  (1.02,0.720) +- (0,0.0032) (1.05,0.921) +- (0,0.0019)
  (1.10,0.997) +- (0,0.0004) (1.15,1.000) +- (0,0.0001)
  (1.20,1.000) +- (0,0.0001)};
\addlegendentry{exact parameters}
\addplot[red,mark=square*,mark size=1.6pt,
  error bars/.cd, y dir=both, y explicit] coordinates {
  (1.02,0.589) +- (0,0.0035) (1.05,0.722) +- (0,0.0032)
  (1.10,0.878) +- (0,0.0023) (1.15,0.957) +- (0,0.0014)
  (1.20,0.988) +- (0,0.0008)};
\addlegendentry{$w{=}0.10$ screen}
\end{axis}
\end{tikzpicture}
\end{minipage}\hfill
\begin{minipage}[t]{0.485\linewidth}
\centering
\begin{tikzpicture}
\begin{axis}[
  width=0.86\linewidth, height=0.70\linewidth,
  xlabel={uncertainty width $w$},
  ylabel={selection probability},
  xmin=0, xmax=0.16, ymin=0.75, ymax=1.02,
  xtick={0,0.05,0.10,0.15}, xticklabels={$0$,$0.05$,$0.10$,$0.15$},
  scaled x ticks=false,
  legend style={at={(0.5,-0.40)}, anchor=north, font=\footnotesize,
    draw=gray!45, fill=white, legend columns=-1},
  tick label style={font=\footnotesize}, label style={font=\footnotesize},
  error bars/error bar style={line width=0.5pt},
]
\addplot[gray!70,dashed,domain=0:0.16] {0.8};
\addlegendentry{benchmark $0.8$}
\addplot[blue,mark=*,mark size=1.5pt,
  error bars/.cd, y dir=both, y explicit] coordinates {
  (0.010,0.995) +- (0,0.0011) (0.015,0.994) +- (0,0.0012)
  (0.020,0.994) +- (0,0.0012) (0.030,0.989) +- (0,0.0017)
  (0.050,0.970) +- (0,0.0027) (0.070,0.939) +- (0,0.0038)
  (0.100,0.882) +- (0,0.0051) (0.150,0.794) +- (0,0.0064)};
\addlegendentry{vs.\ $J_{1.1}$}
\end{axis}
\end{tikzpicture}
\end{minipage}
\caption{Primary single-exchange sweep. Left: probability of selecting~$J$ over a
fixed rival~$J_p$ when~$J$ generates the data, with all boundaries and weight
ratios known exactly and under the $w=0.10$ representation-uncertainty screen.
Right: the corresponding screening probability against~$J_{1.1}$ versus nuisance
width~$w$ (Supplemental~S9.3). Error bars are $\pm1$ binomial Monte Carlo
standard error ($N=2\times10^4$ left, $N=4000$ right) and are smaller than the
markers at several points. The $0.8$ line is a design benchmark for comparing
conditions. The plotted quantity is the simulated probability of selecting the
data-generating model, which is distinct from the rejection rate of a
level-controlled test.}
\label{fig:primary-sweep}
\end{figure}
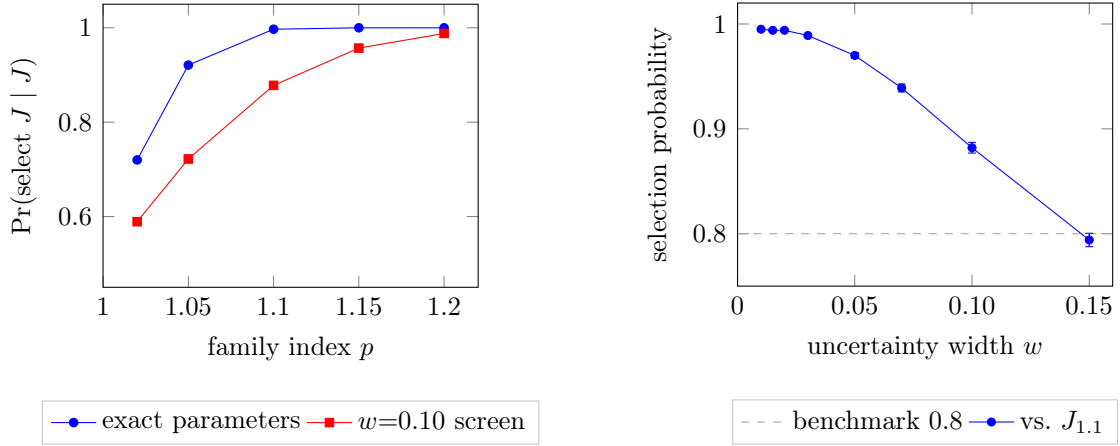
\FloatBarrier

\paragraph{Feasibility and scope}
The design calls for boundary ratios $a=e^u\approx7.4$, $20$, and~$54.6$.
Ratios of this order occur in representative mammalian-cell concentration
gradients: transmembrane
concentration ratios around $145{:}12\approx12$ for $\mathrm{Na}^+$,
$140{:}4\approx35$ for $\mathrm{K}^+$, and $110{:}10\approx11$ for
$\mathrm{Cl}^-$~\cite{Weiss1996}, i.e.\ $u\approx2.5$, $3.6$, and~$2.4$, all
within the range $2\le u\le4$ spanned by the grid. These gradients are
maintained by active transport and are not passive equilibrium boundaries of
the present model; they are cited only to show that the design's boundary
ratios are of an ordinary physical magnitude, and whether such ratios can be
imposed as controlled pinned boundaries is among the untested requirements of
Section~\ref{sec:outlook}. A $12$-point grid (dropping
$\alpha/\beta=1$) reduces the total number of observations from~$105$ to~$84$ without
discarding any pairwise-selection information. The sweep is a
within-family discriminator: it separates prespecified members of the~$J_p$
family and estimates~$p$. Each two-site
condition separates the non-reciprocal comparators $Q$, $S$, and~$D$ weakly;
the large single-configuration separations for those comparators come from the
asymmetric three-site discriminator (summary below; Supplemental~S10).

\subsection{Secondary discriminator (summary)}
\label{sec:threesite-study}
As a secondary mathematical discriminator, an
asymmetric three-site system separates~$J$ from four fixed comparators
calibrated to unit curvature at the reference:
\[
Q(x)=\tfrac12(\log x)^2,\qquad
S(x)=\tfrac12(x-1)^2,\qquad
D(x)=x\log x-(x-1),
\]
together with an elastic reference~$H_{\mathrm{el}}$ defined by quadratic
penalties on linear-coordinate differences
(full formula in Supplemental~S10). At the specified
operating point $(\alpha,\beta_L,\beta_R,a_L,a_R)=(1,2,1,10,0.2)$ the
nearest comparator is~$D$, with log-location separation
$|v_J-v_D|=0.056$. The observation-noise-only simulated selection probability
is~$0.828$; under the stated representation-uncertainty screening model it falls
to~$\approx0.33$. Supplemental~S10 reports the corresponding comparator
predictions, decision rule, replicate count, and nuisance scan.

\FloatBarrier

\section*{Remark on forward dynamics}
\label{sec:variational}

\noindent
The inverse results of Sections~\ref{sec:ident}--\ref{sec:numerics} use only
stationarity of the reduced cost~$\widetilde C$ and therefore do not depend on
a particular relaxation mechanism. For the reciprocal specialization
$\Psi(t)=\cosh t-1$, the unforced gradient flow
$R\dot y=-\nabla\widetilde C(y)$ (constant positive-definite friction~$R$)
has~$\widetilde C$ as a strict Lyapunov function and converges exponentially to
the unique constrained minimizer. Under the finite-horizon regularity
assumptions stated in Supplemental~S3, the corresponding forced Langevin model
converges to the deterministic flow as~$\varepsilon\to0$; when~$\xi\equiv0$, the
diffusion is reversible with invariant density proportional to
$e^{-\widetilde C/\varepsilon}$. The interval balance, exponential relaxation,
and stochastic realization are collected in Supplemental~S2--S3; see
equation~(S2) and Propositions~S2--S3
(cf.~\cite{Mielke2011,AdamsDirrPeletierZimmer2011,MielkePeletierRenger2014,JordanKinderlehrerOtto1998,Santambrogio2017,Polyak1963,KarimiNutiniSchmidt2016,Khasminskii2012,FreidlinWentzell2012}).

\section{Physical interpretation and scope}
\label{sec:outlook}

Theorems~\ref{thm:response-ident}, \ref{thm:fixed-ratio}, and~\ref{thm:chain}
state what exact response data would determine. They do not establish that any physical system is
governed by a functional of the form~\eqref{eq:graph-functional}, nor that its
generator is~$J$. This section records what has not been validated and the
constitutive and modeling assumptions a physical interpretation would require.

\paragraph{What has not been validated}
(i)~No physical system has been shown to realize the controlled two-site
response map, which requires a free compartment with a restoring term
independent of the reservoir, a single open exchange edge, a pinned boundary,
and an effective weight ratio that is either known and sweepable or held
fixed at a value different from one; (ii)~no
generator, $J$ or otherwise, has been reconstructed from measured equilibria,
so the paper makes no claim about which generator, if any, describes a real
system; (iii)~the observation noise $\sigma=0.05$ and the nuisance widths~$w$
of Section~\ref{sec:numerics} are assumed screening values rather than
instrument characterizations, so the simulated selection probabilities are
design quantities, not predictions of experimental performance; (iv)~the
energetic reading of~$C_{\mathcal{M}}$ and the work balance of
Supplemental~S3 are untested; and (v)~the transmembrane ion ratios quoted in
Section~\ref{sec:numerics} indicate only the order of magnitude of the design
boundary ratios. Each item is a prerequisite for, not a consequence of, the
results of Section~\ref{sec:ident}.

Theorems~\ref{thm:fixed-ratio} and~\ref{thm:chain} weaken the control
requirements, since one fixed and possibly unknown ratio suffices, but they do
not show that the precision requirements are weaker. The fixed-ratio
reconstruction reads~$g$ through responses at progressively smaller boundary
values, where relative location errors are largest, and its numerical
conditioning under perturbed response curves is unresolved
(Section~\ref{sec:fixed-ratio}). The calibration-uncertainty results of
Section~\ref{sec:numerics} concern a different, variable-ratio parametric
design and do not quantify the noise behavior of the fixed-ratio design,
which remains open.

\paragraph{Relation to conventional free energies}
For ideal-mixture concentration transport, standard free-energy expressions are
relative-entropy-like rather than of the
form~$J$~\cite{RaoEsposito2016,Qian2007}; the comparator~$D$ of
Section~\ref{sec:numerics} is representative of that family. Any energy-like
reading of~$C_{\mathcal{M}}$ would therefore be an empirical calibration rather
than a derivation: one would introduce a conversion constant with units of
energy per dimensionless cost unit, estimate it on a quasistatic dataset held
separate from any test interval, and treat the nonlinear form of the
functional, not its scale, as the testable content. As
Supplemental~S4 records, such a constant is identified only relative to a fixed
normalization of the graph weights, so the normalization convention would have to
be specified in advance and never refitted.

\paragraph{Constitutive requirements for a dynamical interpretation}
The dissipation identity recorded in Supplemental equation~(S2) (Section~S3) becomes a physical work
balance only after $\xi$ is identified with a measurable generalized force,
$\xi\cdot\dot y$ with physical power, and $\dot y^{\top}R\dot y$ with physical
dissipation. Each identification is an assumption about a particular system, to
be supplied by that realization. In particular the boundary log-ratio~$u$ and any external
tilt~$\xi$ play distinct roles and must not be double-counted: a maintained
reservoir enters through the edge term and shifts the passive equilibrium, which
is a boundary condition rather than an applied force.

\paragraph{Modeling requirements}
Four modeling commitments require independent justification in a physical
realization. First, the pairwise homogeneous ansatz of
Corollary~\ref{thm:graph-propagation} imposes one common local shape on every
site and edge term; a system in which site and edge deviations are penalized
differently falls outside it. Second, the free-site term
$\alpha\Psi(v)$ requires a restoring mechanism independent of the reservoir
edge, since passive exchange with a single reservoir alone would equilibrate at
the boundary. Representing a binding, elastic, or partitioning contribution
through the one-coordinate potential requires a measured equilibrium
relation and demonstrated timescale separation. Third, the one-coordinate
reduction assumes a single open exchange current. With two simultaneous edge
currents the physical power cannot be reconstructed from the rate of change of a
single amount and one effective affinity, and two maintained unequal reservoirs
generically sustain a through-current at steady concentration; such
configurations require explicit edge-current coordinates and lie outside the
present model. Fourth, fixing the representation before any outcome is scored
rules out one form of post hoc adjustment; the physical meaningfulness of the
chosen graph, references, and measurement transforms, and their invariance under
equivalent descriptions, are separate questions, so a representation-sensitivity
analysis belongs in any protocol built on this model.

\paragraph{Requirements for a physical realization}
A physical program built on these results would name an analyte and platform,
establish the constitutive relation underlying the free-site term, calibrate
the controlled boundary values and effective weight ratios independently, or
hold a fixed unequal ratio for the designs of Theorems~\ref{thm:fixed-ratio}
and~\ref{thm:chain}, and then measure equilibrium locations over a prescribed
sweep. Those response
data could then be used to reconstruct the generator, or in a finite
parametric setting to test candidate generators such as~$J$. Energetic
calibration and tests of the supplemental work balance would constitute a
separate downstream question.

\section{Discussion and conclusions}
\label{sec:discussion}

The main results are constructive. For a shared generator in
class~\eqref{eq:Psi-class}, equilibrium locations under controlled boundary
values and weight ratios determine the generator up to a positive scale,
through the explicit inversion formula $r=g(u-v^\star)/g(v^\star)$
(Theorem~\ref{thm:response-ident}). One boundary sweep at a fixed, possibly
unknown, ratio $r\neq1$ also determines it, through Koenigs' linearization of
a Schr\"oder equation that the response curve supplies
(Theorem~\ref{thm:fixed-ratio}). Calibration of the slope at the origin
removes the scale. Finite target lists recover ratios of~$g$ on a grid, with
explicit error bounds under forward root-finding and under noisy bisection
(Corollary~\ref{cor:finite-query}, Proposition~\ref{prop:bisection}), and
recovered log-ratios are locally Lipschitz stable in the location, boundary,
and ratio data (Corollary~\ref{cor:stability}). The results transfer to a
single free vertex of a finite graph under the pinned-neighbor protocol
(Corollary~\ref{cor:graph-reduction}), and the fixed-ratio identification
transfers to a chain with two free vertices, whose leaf sees its free neighbor
as a boundary, under an explicit sweep condition and unequal leaf weights
(Theorem~\ref{thm:chain}).
Relative to inverse optimization, the designed equilibria replace a set of
consistent objectives by an exact one-parameter ambiguity. Relative to
boundary-response inverse problems, the data are locations only and the
unknown is a single shared law of one variable; the local Lipschitz estimate
reflects that different problem and is not a statement about Calder\'on
stability.

A notable feature of the reconstruction is that the observable is only the
equilibrium location: neither the value of the objective nor a boundary
flux is required. Variation of the controlled weight ratio converts these
locations into ratios of $g=\Psi'$, while variation of the boundary probes
different arguments of~$g$. At a fixed unequal ratio the boundary sweep alone
suffices, but the information is then reached through an iteration toward the
origin. The price for this reduced observable is the global scale
non-identifiability and the need for accurate control of the inputs: by
Corollary~\ref{cor:stability}, boundary errors enter the recovered log-ratios
at the same order as location errors.

The reciprocal cost~$J$ enters only as an independently motivated candidate.
The synthetic sweep shows that, within the parametric family~$J_p$ and under
the stated models, a $15$-condition design selects~$J$ over~$J_{1.1}$ with
simulated probability~$0.997$ under exact parameters and about~$0.88$ under the
$w=0.10$ errors-in-variables screen, and that input-calibration uncertainty,
not observation noise, limits precision; these are design properties of the
simulation (level~(b) of Section~\ref{sec:levels}).

\paragraph{Open problems}
\begin{enumerate}
  \item \emph{General graphs with several free vertices.}
    Corollary~\ref{cor:graph-reduction} treats one free vertex and
    Theorem~\ref{thm:chain} a chain in which a free leaf sees an observed,
    sweeping neighbor. On a general graph each stationarity equation couples
    several values of~$g$ with unknown weights. Whether location data
    determine~$g$, and under what graph conditions, is open, as is the case
    of a leaf whose neighbor's coordinate does not approach the reference
    value.
  \item \emph{Unknown and varying weights.} Theorem~\ref{thm:fixed-ratio}
    identifies one fixed unknown ratio together with~$g$, and
    Theorem~\ref{thm:chain} needs no knowledge of the weights. When unknown
    ratios drift within a sweep, or differ between individual equilibria
    rather than between complete sweeps, or when the individual weights are
    sought, weights and generator must be recovered jointly. The scale ambiguity then interacts with weight rescalings
    (Supplemental~S4), and the identifiable set is not characterized here.
  \item \emph{Stability of the fixed-ratio reconstruction.} The
    factor~$\varrho^{n}$ in~\eqref{eq:koenigs} amplifies perturbations in the
    terminal iterate, errors in the response-derived map~$T$ can accumulate
    along the orbit, and truncating the iteration leaves a linearization bias.
    A stability estimate under perturbed response curves, a truncation rule,
    or a regularized fixed-ratio reconstruction is open. Continuous-dependence
    results for Schr\"oder's and Julia's equations on the given
    map~\cite{KuczmaChoczewskiGer1990,AlfaroVigo2020} are a natural starting
    point; what is needed is a bound in terms of perturbations of the
    response curve, which enter~$T$ through the inverse~$\ell^{-1}$.
  \item \emph{Finite noisy data.} Nonparametric estimation of~$g$ from
    finitely many noisy equilibria with uncertain controls, including
    estimators, regularization, convergence rates, data-driven bounds on the
    constants of Corollary~\ref{cor:stability} and
    Proposition~\ref{prop:bisection}, and the correlated errors-in-variables
    models of Section~\ref{sec:numerics}, is not addressed. The
    inverse-variational-inequality estimators
    of~\cite{Bertsimas2015,ThaiBayen2018} are natural candidates.
  \item \emph{Stability of the calibration.} The constants of
    Corollary~\ref{cor:stability} degenerate near the origin, where the
    calibration slope is read; a quantitative stability estimate for the
    absolute scale is open.
  \item \emph{Physical realization.} Section~\ref{sec:outlook} lists the
    constitutive calibrations, and the items not yet validated, that connect
    the inverse framework to a realized system.
\end{enumerate}

\section*{Data availability}

No measured datasets were created or analyzed in this study; every numerical
value in Section~\ref{sec:numerics} and Supplemental~S9, S10, and~S12 follows
from the fixed equations and observation models stated there. The Supplemental Material
and the reproducibility files are distributed as ancillary files with this
arXiv record (directory \texttt{anc/}). The directory
\texttt{anc/reproducibility/} contains the Python entry points
\texttt{RS\_reciprocal\_graph\_numerics.py},
\texttt{RS\_uncertainty\_and\_sweep.py},
\texttt{independent\_check\_numerics.py}, and
\texttt{fixed\_ratio\_reconstruction\_check.py}, which checks
Theorems~\ref{thm:fixed-ratio} and~\ref{thm:chain} and
Proposition~\ref{prop:bisection} numerically; pinned dependencies in
\texttt{requirements.txt} (NumPy~2.4.4, SciPy~1.17.1, pytest~8.3.5);
\texttt{expected\_outputs.txt}; regression tests under \texttt{tests/}; a
README; and an MIT license. Selection and screening simulations use
seed~\texttt{20260713}, estimator simulations seed~\texttt{20260730}, and the
bisection checks seed~\texttt{20261023}. With the pinned versions, the four
scripts reproduce
\texttt{expected\_outputs.txt} line for line (verified on CPython~3.11.4 and
3.11.10), including every value in Table~\ref{tab:p-selection} and
Figure~\ref{fig:primary-sweep}; with other library versions, agreement is
expected within binomial Monte Carlo error. Supplemental~S11 records the
complete reproduction commands. Deferred proofs and numerical details are in
Supplemental~S1--S12.

\section*{Author contributions}

\textbf{Jonathan Washburn}: Conceptualization, Investigation, Methodology,
Resources, Software, Supervision, Validation, Writing -- original draft,
Writing -- review \& editing.
\textbf{Megan Simons}: Formal analysis, Investigation, Project administration,
Validation, Visualization, Writing -- review \& editing.

\section*{Funding}

This research did not receive any specific grant from funding agencies in the
public, commercial, or not-for-profit sectors.

\section*{Conflict of interest}

The authors declare that they have no competing interests.

\section*{Acknowledgments}

We thank colleagues at the Recognition Physics Institute (Elshad Allahyarov,
Sebastian Pardo-Guerra, Anil Thapa, Dylan Funk, Philip Beltracchi, and Emma
Tully) for discussion and review, particularly of the inverse-problem
formulation, three-site design, and physical interpretation.

\end{document}